\documentclass[a4paper,english,10pt,oneside]{article}

\usepackage[english]{babel}
\usepackage{amsmath}
\usepackage{amssymb}
\usepackage{wasysym}
\usepackage{amsthm}  
\usepackage{tikz}
\usepackage{pgfplots}

\usepackage{graphicx}
\usepackage{mathtools}
 \usepackage[margin=1in]{geometry}
\usepackage{booktabs}
\usepackage{xspace}
\usepackage[hidelinks]{hyperref}
\usepackage[font=small]{caption}

\usepackage{subcaption}
\usepackage{todonotes}

\usepackage{thmtools}
\usepackage{thm-restate}

\usepackage{chngcntr}
\usepackage{apptools}
\AtAppendix{\numberwithin{lemma}{section}}
\AtAppendix{\numberwithin{proposition}{section}}

\usepackage{cleveref}
\usepackage[percent]{overpic}

\graphicspath{{pictures/}}

\definecolor{darkblue}{rgb}{0,0,.5}
\DeclareMathOperator{\conv}{conv} 
\DeclareMathOperator{\clconv}{\overline{conv}} 
\DeclareMathOperator{\inte}{int} 

\newcommand{\st}{\,:\,}

\newcommand{\genericquadopen}{S}
\newcommand{\genericquadclosed}{T}
\newcommand{\PDLC}{PDLC}

\usepackage{xifthen}
\ifthenelse{\isundefined{\T}}{%
  \newcommand{\T}{\mathsf{T}}
  }{%
  \renewcommand{\T}{\mathsf{T}}
}

\newtheorem{theorem}{Theorem}
\newtheorem{myproposition}{Proposition}

\newtheorem{mylemma}{Lemma}
\newtheorem{example}{Example}

\newtheorem{definition}{Definition}
\newtheorem{claim}{Claim}

\author{Santanu S. Dey\footnote{Georgia Institute of Technology, Atlanta, GA, USA, santanu.dey@isye.gatech.edu}  \and Gonzalo Mu\~noz\footnote{Universidad de O'Higgins, Rancagua, Chile, gonzalo.munoz@uoh.cl} \and Felipe Serrano\footnote{$\text{I}^2$DAMO GmbH, Engleralle 19, 14169 Berlin, Germany, serrano@i2damo.de }}

\title{On obtaining the convex hull of quadratic inequalities via aggregations%
\footnote{Gonzalo Mu\~noz would like to thank the support of the Research and Development Agency of Chile (ANID) through Fondecyt grant number 11190515. Santanu S. Dey would like to gratefully acknowledge the support of the grant N000141912323 from ONR.}}

\begin{document}
\maketitle
\begin{abstract}
A classical approach for obtaining valid inequalities for a set involves weighted aggregations of the inequalities that describe such set. When the set is described by linear inequalities, thanks to the Farkas lemma, we know that every valid inequality can be obtained using aggregations. When the inequalities describing the set are two quadratics, Yildiran \cite{yildiran2009convex} showed that the convex hull of the set is given by at most two aggregated inequalities. In this work, we study the case of a set described by three or more quadratic inequalities. We show that, under technical assumptions, the convex hull of a set described by three quadratic inequalities can be obtained via (potentially infinitely many) aggregated inequalities. We also show, through counterexamples, that it is unlikely to have a similar result if either the technical conditions are relaxed, or if we consider four or more inequalities.
\end{abstract}



\section{Introduction}\label{sec:intro}

Given a feasible region described by two or more constraints, a common approach for obtaining relaxations of the set is to consider weighted aggregations,  that is, the process of obtaining a new inequality by re-scaling the constraints by scalar weights and then adding the scaled constraints together. We call this approach \emph{aggregation}.
In the case of a nonempty set described by a finite number of linear inequalities we know, thanks to the Farkas Lemma, that any implied inequality can be obtained via an aggregation. Aggregations have also been studied in the context of integer programming (for example~\cite{bodur2018aggregation}) to obtain cutting-planes and in mixed-integer nonlinear programming (for example~\cite{muller2020generalized}) to obtain better dual bounds. 

In this paper, we are interested in understanding the strength of aggregations in order to obtain the convex hull of sets defined by \emph{quadratic constraints}. 
In ~\cite{yildiran2009convex,Burer2016,Modaresi2017}, the authors have shown that the convex hull of \emph{two} quadratic inequality constraints can be obtained as the intersection of a finite number of aggregated constraints (in fact, two). Each of these aggregated constraints may not be convex on its own, but their intersection gives the convex hull. The main question we ask in this paper is whether such an aggregation technique can be shown to deliver the convex hull of sets defined using more than two quadratic constraints.

Our key result is to show, that under a nontrivial technical condition, a similar result to the case of sets described by two quadratic constraints can be obtained for the case of three quadratic constraints ---that is, the convex hull of a set described by three quadratic constraints can be obtained as the intersection of aggregated constraints where, individually, these aggregated constraints may not be convex.
Overall, we follow closely the presentation style in Yildiran~\cite{yildiran2009convex} and follow a similar high-level proof strategy. 

The key to our approach involves proving a three-quadratic-constraints S-lemma~\cite{yakubovich1977,polik2007survey}. An S-lemma for three quadratics is known, however, for our purposes, we require a different version of it which, to the best of our knowledge, has not been proved. Our S-lemma result is based on a theorem due to Barvinok~\cite{barvinok2001remark}.

We also show, via examples, that this result provides a reasonable demarcation of conditions that allow obtaining the convex hull of quadratic constraints via aggregated constraints. In particular, we present an example with three quadratic constraints where the necessary technical condition for our result does not hold, and the convex hull is not obtained using aggregated constraints. We also present an example with four quadratic constraints where the convex hull is not obtained using aggregated constraints even though the necessary technical conditions hold. 

\subsection{Literature review}
The results presented in this paper contribute to the current literature on understanding the structure of the convex hull of simple sets described by quadratic constraints. Recently, the convex hull of one quadratic constraint intersected with a polytope, and other cutting-plane generation techniques for this set have been studied in~\cite{tawarmalani2010strong, DeySantana2018,asterquad,bienstock2020outer,munoz2020maximal,GDR2021}.
The convex hull for two quadratic constraints and related sets have been studied in~\cite{yildiran2009convex,Burer2016,Modaresi2017,dey2019convexifications}. 

Other connections to previous literature are convex hull results for sets related to the so-called extended trust-region problem
\cite{ye2003new,burer2013second,burer2015trust,burer2015gentle,anstreicher2017kronecker,belotti2013families}. 
Another connection is given by the general conditions for the SDP relaxation being tight/giving the convex hull, which have been studied in~\cite{wang2021tightness,wang2020convex,burer2019exact,argue2020necessary}. It is important to mention that the results in \cite{wang2021tightness}, providing conditions for the tightness of the SDP relaxation, involve the study of the aggregation of quadratic inequalities, as in our case. However, the results are of a slightly different nature; for instance, \Cref{example:1-2} below illustrates a case where the convex hull can be obtained via aggregations, but the SDP relaxation is not tight. 

\subsection{Notation}
For a positive integer $n$, we denote the set $\{1, \dots, n\}$ by $[n]$.  Let $\mathbb{S}^n$ denote the space of $n\times n$ symmetric matrices, and $\mathbb{S}^n_+$ denote the space of positive semi-definite matrices. We denote the fact that $A \in \mathbb{S}^n_{+}$  as $A \succeq 0$, and the fact that $A \in \mathbb{S}^n$  is a positive-definite matrix as $A \succ 0$. Given a set $S$, we denote its convex hull, interior, and closure as $\textup{conv}(S)$, $\inte(S)$, and $\bar{S}$ respectively. 

Given a set $S$ defined by one quadratic constraint, that is, $$S:= \{x \in \mathbb{R}^n: x^{\top}Ax + 2b^{\top}x + c \ \spadesuit \ 0\},$$ where $\spadesuit$ is either $<$ or $\leq$, we let $\nu(S)$ denote the number of negative eigenvalues of $A$.  Given a set $S$ described by $m$ quadratic constraints:
$$S:= \{x \in \mathbb{R}^n: x^{\top}A_ix + 2b_i^{\top}x + c_i \ \spadesuit \ 0, \ i \in [m]\},$$ where $\spadesuit$ is either $<$ or $\leq$ for all the constraints, we  denote the homogenization of this set as $S^h$, that is
$$S^h := \left\{(x, x_{n +1}) \in \mathbb{R}^n\times \mathbb{R}: x^{\top}A_ix + 2x^{\top}b_i x_{n +1} + c_i x_{n+1}^2 \ \spadesuit \ 0, \ i \in [m]\right\}.$$
Given $\lambda \in \mathbb{R}^{m}_{+}$, we let $S_{\lambda}$ to denote the set defined by aggregation of the constraints in $S$ with the weights $\lambda$, that is
$$S_{\lambda} := \left\{x \in \mathbb{R}^n : x^{\top}\left(\sum_{i \in [m]} \lambda_iA_i \right)x + 2x^{\top}\left(\sum_{i \in [m]}\lambda_ib_i\right)  + \sum_{i \in [m]}\lambda_i c_i\ \spadesuit \ 0 \right\}.$$

\subsection{Outline of the paper} In \Cref{sec:main} we present our main results, including examples of the results and counterexamples illustrating the importance of the technical conditions in our results. 
%
%
In \Cref{sec:conclusion} we provide open questions and conclusions from our main results.
\Cref{sec:proofslemma,sec:proofthmopencase,sec:proofclosedcase,sec:counterexproofs} present the proofs of each of our results. In each one of these sections we provide the preliminary results needed for each proof. 

\section{Main results}\label{sec:main}

In this section, we provide our main results along with the necessary background. We also provide examples illustrating the main results and counter-examples showing the importance of our conditions. All proofs are presented in \Cref{sec:proofslemma} and onwards.

\subsection{Known S-lemmas and a new variant}\label{sec:prelimslem}

At the core of the main results of our work is the S-lemma, which has a rich history. In its most modern form, it was first proven by Yakubovich~\cite{yakubovich1977}. See the excellent survey by P\'olik and Terlaky~\cite{polik2007survey}. 

The following version of S-lemma was used by Yildiran~\cite{yildiran2009convex} in his proof of the convex hull result for two quadratics. 
\begin{theorem}[S-lemma for two quadratics, used in \cite{yildiran2009convex}]\label{thm:Slemma2}
Let $g_1,g_2:\mathbb{R}^n \rightarrow \mathbb{R}$ be homogeneous quadratic functions:
\[g_i(x) = x^\top Q_i x.\]
Then,
\[\{x\in\mathbb{R}^n \,|\, g_i(x) < 0, i \in [2] \} = \emptyset \Longleftrightarrow \exists \lambda \in \mathbb{R}^2_+ \setminus \{0\},\, \sum_{i=1}^2 \lambda_i Q_i \succeq 0.\]
\end{theorem}
In order to obtain a convex hull result for three quadratics, a similar theorem would be ideal. The S-lemma does have variants that include three quadratic inequalities, such as the following. 

\begin{myproposition}[Proposition 3.6, S-lemma survey~\cite{polik2007survey}]\label{prop:slemmaasy}
Let $n \geq 3$ and $g_0,g_1,g_2: \mathbb{R}^n \rightarrow \mathbb{R}$ be homogeneous quadratic functions:
\[g_i(x) = x^\top Q_i x \]
Assume there is $\hat{x}\in \mathbb{R}^n$ such that $g_1(\hat{x})<0, g_2(\hat x) < 0$, and that there is a linear combination of $Q_0, Q_1, Q_2$ that is positive definite.  Then
\begin{align*}
&\quad \{x\in \mathbb{R}^n \st g_0(x) <0, g_1(x) \leq 0, g_2(x)\leq 0\} = \emptyset  \\
\Longleftrightarrow &\quad  \exists\, (y_1,y_2)\in \mathbb{R}^2_+\quad g_0(x) + y_1g_1(x) + y_2g_2(x) \geq 0 \quad \forall x\in \mathbb{R}^n 
\end{align*}
\end{myproposition}

Note that this S-lemma is ``asymmetrical'', in the sense that one inequality is singled out from the other two. In \Cref{thm:Slemma2} this is not the case, as both inequalities are treated symmetrically. 
In our development, a symmetric version of the \Cref{prop:slemmaasy} is desirable, and we thus prove the following.

\begin{restatable}{mylemma}{Slemmathree}\label{thm:Slemma3}
Let $n \geq 3$ and let $g_1,g_2,g_3: \mathbb{R}^n \rightarrow \mathbb{R}$ be homogeneous quadratic functions:
\[g_i(x) = x^\top Q_i x.\]
Assuming there is a linear combination of $Q_1, Q_2, Q_3$ that is positive definite, the following equivalence holds
\[\{x\in \mathbb{R}^n:  g_i(x) < 0,\, i\in [3]\} = \emptyset \Longleftrightarrow \exists \lambda \in \mathbb{R}^3_+ \setminus \{0\},\, \sum_{i=1}^3 \lambda_i Q_i \succeq 0.\]
\end{restatable}
It is important to mention that in the case of two quadratics an asymmetrical version of  \Cref{thm:Slemma2} is well known, and the equivalence between both versions can easily be established. However, in the case of three quadratics, we do not see a direct way of proving \Cref{thm:Slemma3} from \Cref{prop:slemmaasy} and thus we present a direct proof.


Just as one can prove the Farkas Lemma as a consequence of strong duality for linear programming, one proof of the original two-quadratic-constraints S-lemma can be seen as the consequence of strong duality for semidefinite programming (SDP) and a `rank reduction' result. With the latter, feasibility of the primal SDP implies the existence of a rank-one solution for the SDP ---this yields feasibility of the original quadratic constraints. 

For the two-quadratic-constraints S-lemma, the classical result of Pataki~\cite{pataki1998rank} (see \Cref{thm:Pataki} in \Cref{sec:prelimopt}) suffices to accomplish the rank reduction. 
In the case of three-quadratic-constraints S-lemma, Pataki's result does not suffice and we rely on a similar result due to Barvinok~\cite{barvinok2001remark} that holds under a boundedness condition (see \Cref{thm:Barvinok} in \Cref{sec:prelimopt}). The proof of \Cref{thm:Slemma3} can be found in \Cref{sec:proofslemma} based on the outline presented above.

\subsection{The convex hull of three quadratic constraints: open case}

The main result of this paper provides sufficient conditions for the convex hull of a set defined by three quadratic inequalities to be given by aggregations. Specifically:
\begin{restatable}{theorem}{thmopen}\label{thm:opencase}
Let $n \geq 3$ and
\[ \genericquadopen = \left\{ x \in \mathbb{R}^n : [x\quad 1] \begin{bmatrix} A_i & b_i \\ b^{\top}_i & c_i  \end{bmatrix} \begin{bmatrix} x \\ 1 \end{bmatrix} < 0,\ i \in [3] \right\} . \] 
Assume 
\begin{itemize}
\item (Positive definite linear combination, or \PDLC{}) There exists $\theta\in \mathbb{R}^3$ such that 
\[\sum_{i=1}^3 \theta_i \begin{bmatrix} A_i & b_i \\ b^{\top}_i & c_i  \end{bmatrix} \succ 0.\]
\item (Non-trivial convex hull) $\textup{conv}(\genericquadopen)\neq \mathbb{R}^n.$
\end{itemize}
Let
\[\Omega \coloneqq \left\{\lambda \in \mathbb{R}^3_+\st \genericquadopen_\lambda \supseteq \conv(\genericquadopen) \textup{ and } \nu(\genericquadopen_\lambda) \leq 1 \right\}, \] 
where $\genericquadopen_\lambda = \left\{ x\in \mathbb{R}^n: [x\quad 1] \left(\sum_{i =1}^3 \lambda_i\begin{bmatrix} A_i & b_i \\ b^{\top}_i & c_i  \end{bmatrix}\right) \begin{bmatrix} x \\ 1 \end{bmatrix} < 0 \right\}$
and $\nu(\genericquadopen_\lambda)$ is the number of negative eigenvalues of $\sum_{i=1}^3 \lambda_i A_i$.
Then
\[\conv(\genericquadopen) = \bigcap_{\lambda\in \Omega} \genericquadopen_\lambda. \]
\end{restatable}
Our proof of \Cref{thm:opencase} in presented in \Cref{sec:proofthmopencase} and follows the following arguments.
We consider any point $\hat{x}\not\in \textup{conv}(\genericquadopen)$ and set our task to proving that there is a $\lambda \in \Omega$ such that $\hat{x}\not\in \genericquadopen_{\lambda}$. 
We begin by selecting a hyperplane separating $\hat{x}$ from $\textup{conv}(\genericquadopen)$ and homogenizing both this hyperplane and the set $\genericquadopen$. Effectively, in the linear subspace defined by the homogenized hyperplane, the three homogeneous quadratic constraints define an infeasible set (\Cref{prop:validineqseparates}, \Cref{sec:proofopencase1}). From here, we apply the S-lemma (\Cref{thm:Slemma3}) and obtain a $\lambda \in \mathbb{R}^3_+$ such that the quadratic form obtained by aggregating the homogeneous quadratic constraints with $\lambda$ does not intersect the homogenized hyperplane (\Cref{thm:separation}, \Cref{sec:proofopencase2}).
%
%
Finally, we complete the proof of \Cref{thm:opencase} (\Cref{sec:puttingittogetheropencase}) by showing that, after a ``dehomogenization'', this implies (i) $\hat{x}\not\in \genericquadopen_\lambda$ and (ii) $\lambda \in \Omega$.


We note that, even though the high-level approach of our proof of \Cref{thm:opencase} is similar to the one by Yildiran~\cite{yildiran2009convex}, the proof itself is simpler.  Yildiran uses the S-lemma in its two-quadratic version, but also heavily depends on several structural results regarding the pencil of two quadratics ---our proof essentially only uses the S-lemma.  
However, the result by Yildiran is stronger; our simplification of the proof, and the lack of existence of the exact same version of S-lemma for three quadratic constraints, lead to some important differences between \Cref{thm:opencase} and the analogous result in \cite{yildiran2009convex}: 
\begin{itemize}
\item \Cref{thm:opencase} requires the \PDLC{} condition, unlike Yildiran's result which does not need any such condition. This condition is required since we need a similar condition for the S-lemma (\Cref{thm:Slemma3}).
\item The set $\Omega$ may not be finite, unlike Yildiran's result which shows that one only needs two aggregations for constructing the convex hull of two quadratic inequalities. In our case, we currently do not know if we require only a finite subset of $\Omega$ to obtain $\textup{conv}(S)$ for the three quadratic constraints case.
\end{itemize}



The next example shows an application of \Cref{thm:opencase}.
\begin{example}\label{example:1}
Consider the following set:
$$\genericquadopen = \left\{ x \in \mathbb{R}^3 : [x\quad 1] \begin{bmatrix} A_i & b_i \\ b^{\top}_i & c_i  \end{bmatrix} \begin{bmatrix} x \\ 1 \end{bmatrix} < 0, \, i \in [3] \right\},$$
where, 
\begin{itemize}
\item $A_1 =  \left[ \begin{array}{ccc} 1& 0 & 0\\ 0 & 1 & 0 \\ 0 & 0 & 0 \end{array}\right] $, $b^{\top}_1 = [0 \ 0 \ 0]$, $c_1 = -2 $ 
\item $A_2 = \left[ \begin{array}{ccc} -1& 0 & 0\\ 0 & -1 & 0 \\ 0 & 0 & 0 \end{array} \right] $, $b^{\top}_2 = [0 \ 0 \ 0]$, $c_2 = 1$ 
\item $A_3 = \left[ \begin{array}{ccc} -1& 0 & 0\\ 0 & 1 & 0 \\ 0 & 0 & 1 \end{array}\right] $, $b^{\top}_3 = [3 \ 0 \ 0]$, $c_3 = 0$ 
\end{itemize}
Note that $$-12\cdot\begin{bmatrix} A_1 & b_1 \\ b^{\top}_1 & c_1  \end{bmatrix} - 15\cdot\begin{bmatrix} A_2 & b_2 \\ b^{\top}_2 & c_2 \end{bmatrix} +  1\cdot\begin{bmatrix} A_3 & b_3 \\ b^{\top}_3 & c_3  \end{bmatrix}= \left[\begin{array}{cccc}  2& 0& 0& 3\\ 0&4& 0& 0 \\ 0& 0& 1 & 0 \\3 &0 &0 &9\end{array}\right] \succ 0,$$ i.e. the  \PDLC{} condition holds. In addition, it is easy to verify that S is bounded, thus $\textup{conv}(\genericquadopen) \neq \mathbb{R}^3$. Therefore, by \Cref{thm:opencase} the convex hull is given by the intersection of a family of aggregations. Indeed, in this case the convex hull of $S$ is given by 
$$\textup{conv}(\genericquadopen) = \{x : x_1^2 + x_2^2 - 2 < 0, -x_1^2 + x_2^2 + x_3^2 + 6x_1 < 0, -2x_1^2 + x_3^2 + 6x_1 + 1 <0\},$$
where the first two inequalities correspond to the first and the third inequality describing $S$, and the last inequality is an aggregation with $\lambda = (0, 1,1)$. 

Note also that the matrices $A_1,A_3, A_2 + A_3$, each corresponding to the quadratic part of a constraint in the description of $\textup{conv}(\genericquadopen)$, all have at most one negative eigenvalue.
 Moreover, the second and third constraints are not describing convex sets on their own, but combined with the first constraint they yield a $\conv(\genericquadopen)$.
In \Cref{fig:finiteaggregation} we show $\genericquadopen$ and $\conv(S)$ for this example.

\begin{figure}
    \centering
    \includegraphics[scale=0.25]{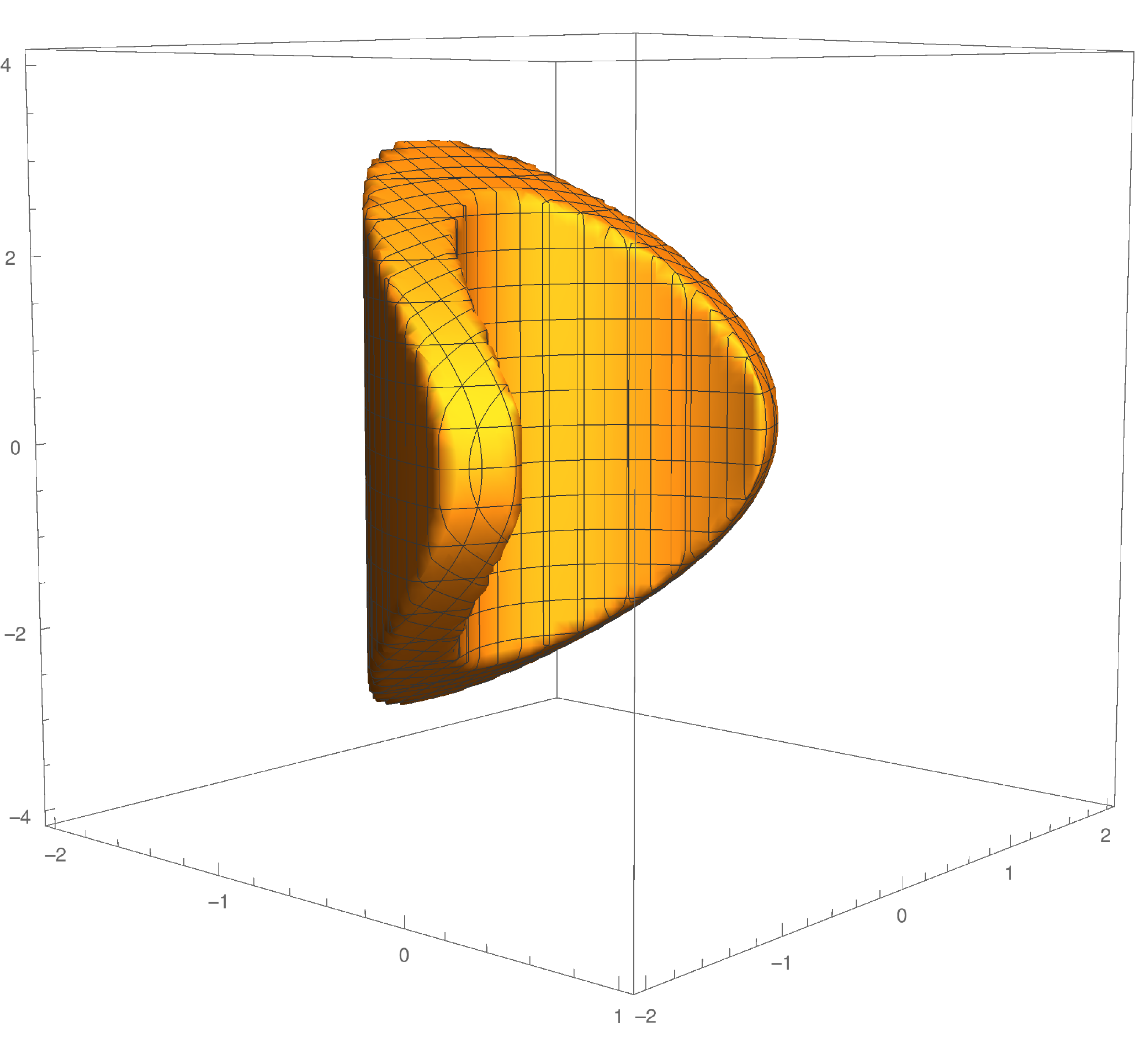}\hspace{.4cm}
    \includegraphics[scale=0.25]{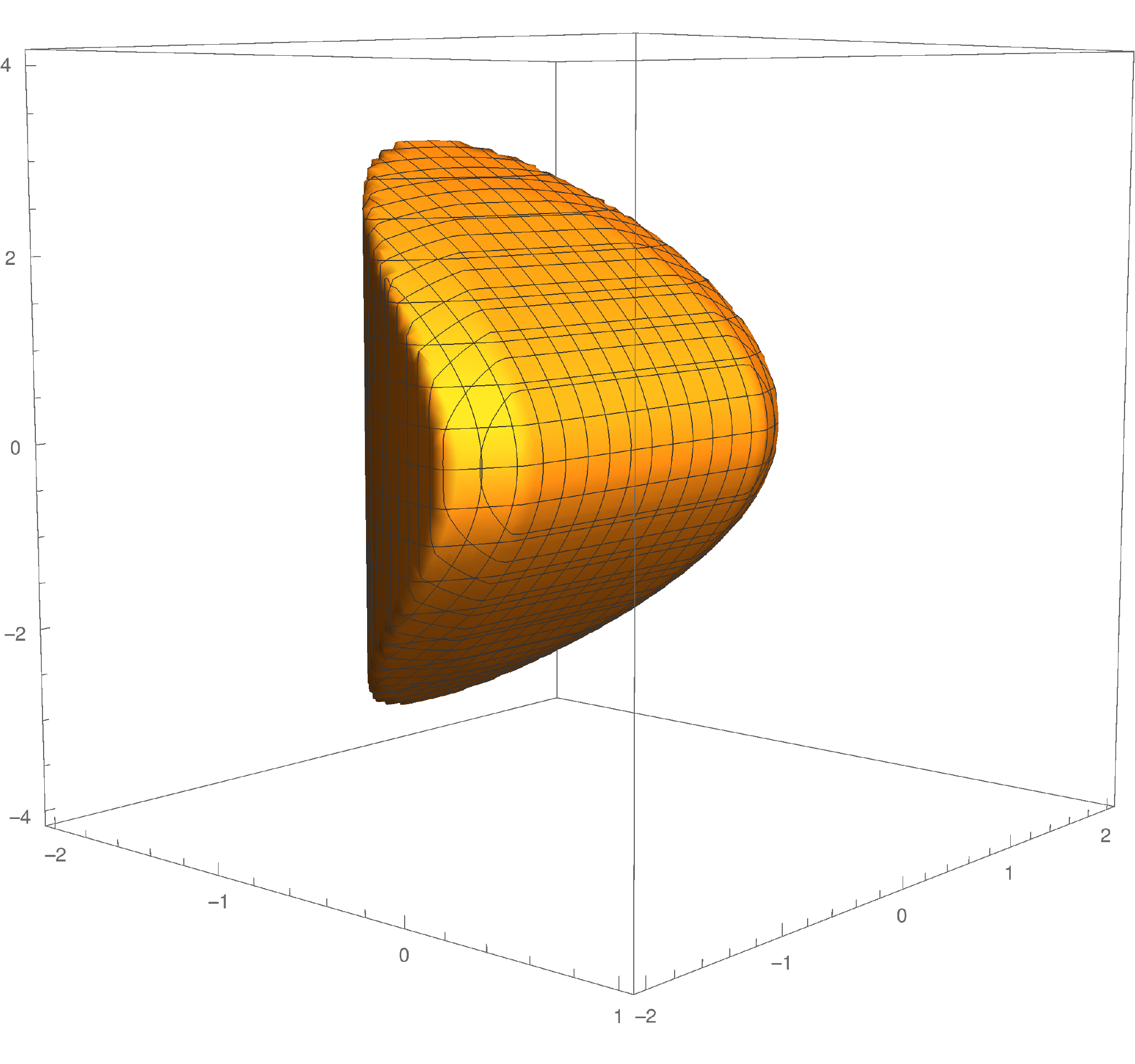}
    \caption{Plots of sets $\genericquadopen$ (left) and $\conv(\genericquadopen)$ (right) for \Cref{example:1}. In this example, the convex hull is obtained via three aggregations of the three quadratic inequalities describing $\genericquadopen$.}
    \label{fig:finiteaggregation}
\end{figure}

Finally, we note that in this case restricting to any pair of constraints of $\genericquadopen$, and aggregations thereof, does not yield the convex hull. 

\end{example}

\subsection{Counterexamples}

The next examples evaluate how important are the requirements of \Cref{thm:opencase}; they show that this theorem is indeed a reasonable demarcation of cases where aggregation can yield the desired convex hull.
We remark that in the examples that follow,
our claims require formal proofs, which we relegate to \Cref{sec:4quadproof,sec:3quadNotPDLC}.

We begin with an example showing that an extension of \Cref{thm:opencase} to four quadratics is unlikely.
We obtained this example from modifying an example of \cite{polik2007survey} (also see \cite[Exercise 3.58]{NemirovskiNotes}) which the authors used to argue that a generalization of the S-lemma to four inequalities is improbable.
Here, we further elaborate on a variation of it to make it fit our purposes.
\begin{restatable}{myproposition}{fourquadexample}\label{prop:4quadexample}
Consider the following set:
\[\genericquadopen = \left\{ x \in \mathbb{R}^3 : [x\quad 1] \begin{bmatrix} A_i & 0 \\0 & c_i  \end{bmatrix} \begin{bmatrix} x \\ 1 \end{bmatrix} < 0, \ i \in [4] \right\},\]
where, 
\begin{itemize}
\item $A_1 =  \left[ \begin{array}{ccc} 1& 1.1 & 1.1\\ 1.1 & 1 & 1.1 \\ 1.1 & 1.1 & 1 \end{array}\right] $, $c_1 = -1$ 
\item $A_2 = \left[ \begin{array}{ccc} -2.1& 0 & 0\\ 0 & 1 & 0 \\ 0 & 0 & 1 \end{array} \right] $, $c_2 = 0$ 
\item $A_3 = \left[ \begin{array}{ccc} 1& 0 & 0\\ 0 & -2.1 & 0 \\ 0 & 0 & 1 \end{array}\right] $, $c_2 = 0$ 
\item $A_4 = \left[ \begin{array}{ccc} 1& 0 & 0\\ 0 & 1 & 0 \\ 0 & 0 & -2.1 \end{array} \right] $, $c_2 = 0$ 
\end{itemize}
In this case,
\begin{itemize}
    \item \PDLC{} is satisfied.
    \item $\textup{conv}(\genericquadopen) \neq \mathbb{R}^3$
\end{itemize}
However,
\[\conv(S) \subsetneq \bigcap \{ \genericquadopen_\lambda \st \lambda \in \mathbb{R}^3_+,\, \genericquadopen_\lambda \supseteq \conv(\genericquadopen) \}.\]
In particular, \(\conv(S) \neq \bigcap_{\lambda\in \Omega} \genericquadopen_\lambda\).
\end{restatable}
In \Cref{fig:counterex4quads} we show the set $S$ considered in \Cref{prop:4quadexample}. The proof of this proposition can be found in \Cref{sec:4quadproof}.

\begin{figure}
    \centering
    \includegraphics[scale=0.25]{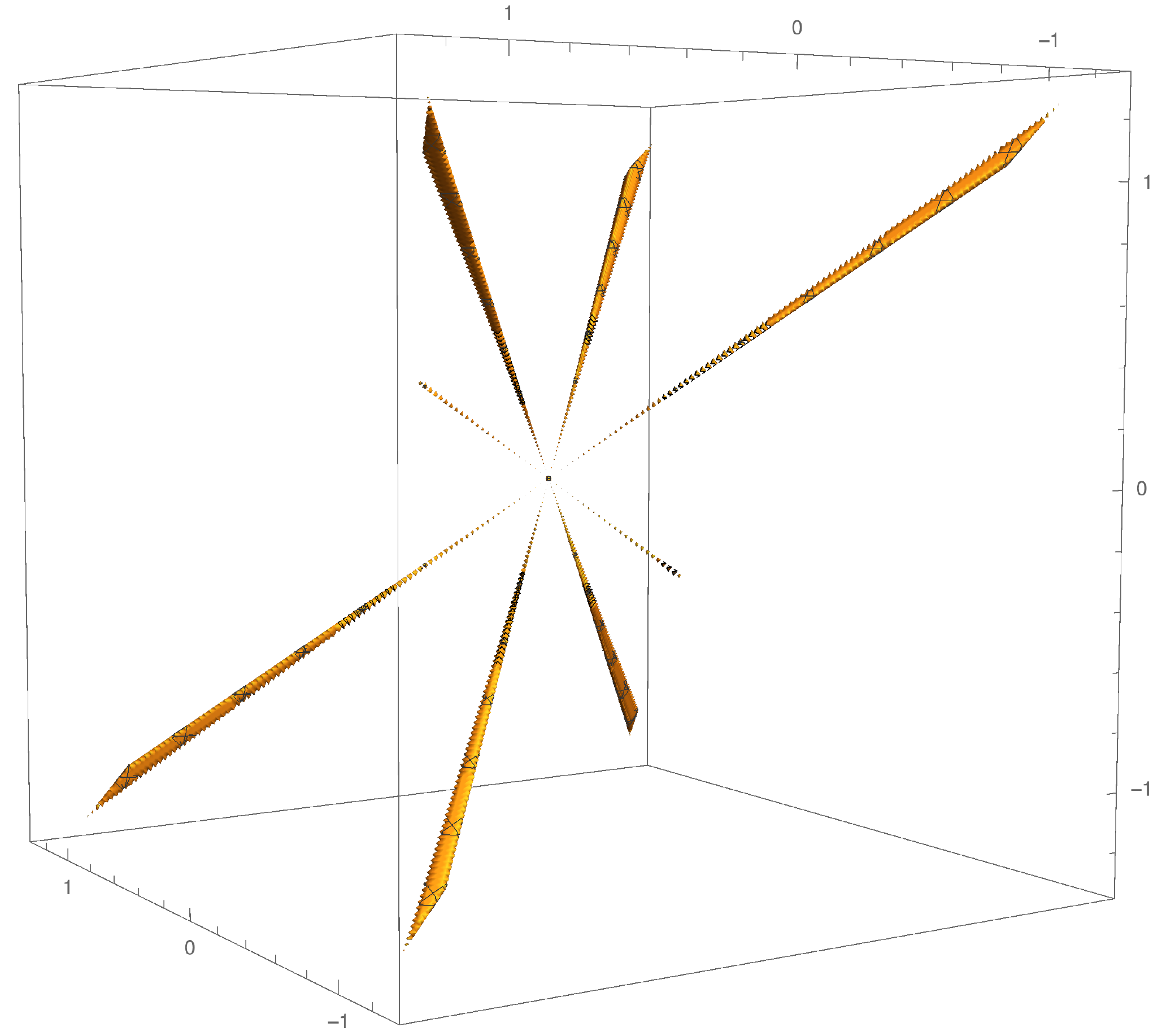}
    \caption{Plot of the set $\genericquadopen$ for \Cref{prop:4quadexample}, defined using 4 quadratic inequalities.}
    \label{fig:counterex4quads}
\end{figure}

The \PDLC{} condition we consider in \Cref{thm:opencase} appears to be quite restrictive, and the reader might wonder how necessary it is. We show next, via an example, that if this condition does not hold then \Cref{thm:opencase} does not necessarily hold. 

\begin{restatable}{myproposition}{threequadnotpdlc}\label{example:3} 
Consider the following set:
$$\genericquadopen = \left\{ x \in \mathbb{R}^3 : [x\quad 1] \begin{bmatrix} A_i & b_i \\ b^{\top}_i & c_i  \end{bmatrix} \begin{bmatrix} x \\ 1 \end{bmatrix} < 0,\, i \in [3] \right\},$$ where, 
\begin{itemize}
\item $A_1 =  \left[ \begin{array}{ccc} 1& 0 & 0\\ 0 & 0 & 0 \\ 0 & 0 & 0 \end{array}\right] $, $b^{\top}_1 = [0 \ 0 \ 0]$, $c_1 = -1 $ 
\item $A_2 = \left[ \begin{array}{ccc} 0 & 0 & 0\\ 0 & 1 & 0 \\ 0 & 0 & 0 \end{array} \right] $, $b^{\top}_2 = [0 \ 0 \ 0]$, $c_2 = -1$ 
\item $A_3 = \left[ \begin{array}{ccc} 0& -1/2 & 0\\ -1/2 & 0 & 0 \\ 0 & 0 & 1 \end{array}\right] $, $b^{\top}_3 = [0 \ 0 \ 0]$, $c_3 = 0$ 
\end{itemize}
In this case,
\begin{itemize}
    \item $\conv(S)\neq \mathbb{R}^3$
    \item \PDLC{} does not hold
\end{itemize}
and additionally, 
\[\conv(S) \subsetneq \bigcap \{ \genericquadopen_\lambda \st \lambda \in \mathbb{R}^3_+,\, \genericquadopen_\lambda \supseteq \conv(\genericquadopen) \}.\]
In particular, \(\conv(S) \neq \bigcap_{\lambda\in \Omega} \genericquadopen_\lambda\).
\end{restatable}

\Cref{fig:counterexPDLC} illustrates the set considered in \Cref{example:3}, and a proof of this result is provided in \Cref{sec:3quadNotPDLC}

\begin{figure}
    \centering
    \includegraphics[scale=0.25]{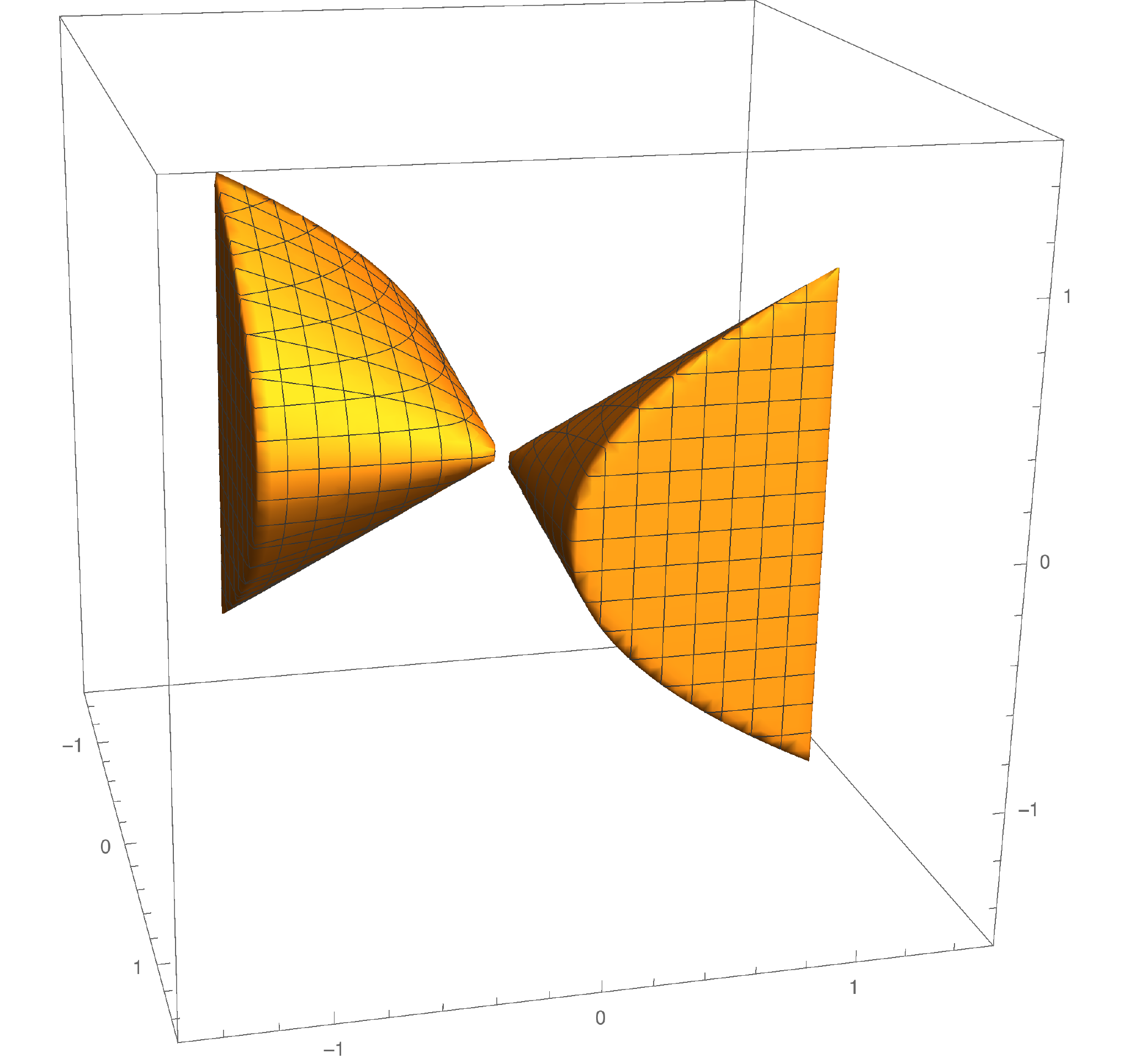}
    \caption{Plot of the set $\genericquadopen$ for \Cref{example:3}. In this example, the \PDLC{} condition does not hold, and the convex hull is not obtained via aggregations.}
    \label{fig:counterexPDLC}
\end{figure}

\subsection{The convex hull of three quadratic constraints: closed case}
In the mathematical programming literature, we usually work with closed sets and constraints defined by inequalities rather than strict inequalities. Therefore, it would be nice to have a version of \Cref{thm:opencase} where we examine a set described by three quadratic (non-strict) inequalities. Here we pursue a direction closely related to that of Modaresi and Vielma~\cite{Modaresi2017} in order to obtain the following result.
\begin{restatable}{theorem}{thmclosed}\label{thm:closedcase}
Let $n \geq 3$ and let
\[\genericquadclosed = \left\{ x \in \mathbb{R}^n : [x\quad 1] \begin{bmatrix} A_i & b_i \\ b^{\top}_i & c_i  \end{bmatrix} \begin{bmatrix} x \\ 1 \end{bmatrix} \leq 0, \ i \in [3] \right\} . \] 
Assume 
\begin{itemize}
\item (Positive definite linear combination, or \PDLC{}) There exists $\theta\in \mathbb{R}^3$ such that 
\[\sum_{i=1}^3 \theta_i \begin{bmatrix} A_i & b_i \\ b^{\top}_i & c_i  \end{bmatrix} \succ 0.\]
\item (Non-trivial convex hull) $\textup{conv}(\genericquadclosed)\neq \mathbb{R}^n.$
\item (No low-dimensional components) $\genericquadclosed \subseteq \overline{\inte(\genericquadclosed)}$.
\end{itemize}
Let
\[\Omega' \coloneqq \{\lambda \in \mathbb{R}^3_+\st \genericquadclosed_\lambda \supseteq \conv(\genericquadclosed) \textup{ and } \nu(\genericquadclosed_\lambda) \leq 1 \}, \]
where $\genericquadclosed_\lambda = \left\{ x\in \mathbb{R}^n: [x\quad 1] \left(\sum_{i =1}^3 \lambda_i\begin{bmatrix} A_i & b_i \\ b^{\top}_i & c_i  \end{bmatrix}\right) \begin{bmatrix} x \\ 1 \end{bmatrix} \leq 0 \right\}$
and $\nu(\genericquadclosed_\lambda)$ is the number of negative eigenvalues of $\sum_{i=1}^3 \lambda_i A_i$.
Then
\[\clconv(\genericquadclosed) = \bigcap_{\lambda\in \Omega'} \genericquadclosed_\lambda.\]
\end{restatable}

At a high level, our proof of \Cref{thm:closedcase} is similar to the proof presented in~\cite{Modaresi2017} for the case of two quadratic constraints. The main difference in the proofs comes from the fact that we do not know $|\Omega|$  is finite (unlike in the case of two quadratic constraints). See \Cref{sec:proofclosedcase} for a proof of \Cref{thm:closedcase}.  

\begin{example}\label{example:1-2}
We note that the `closed' version of \Cref{example:1} serves as an example for the application of \Cref{thm:closedcase}. 

Moreover, we can use this case to argue why our result is of a different nature compared to that of SDP tightness results, such as \cite{wang2021tightness}. We do so by arguing that in this case, even if aggregations yield the convex hull, an SDP may not provide tight bounds.
For example, if we maximize $x_1$ over the resulting closed set $\genericquadclosed$ the optimal value is $< 0$, while the SDP bound is $>0$. We provide detailed derivations of these values (optimal solution value and SDP bound) in \Cref{sec:SDPtightness}.
\end{example}

\subsection{On restricting to a finite subset of $\Omega$}\label{sec:open}
An important open question is to determine if we actually require an infinite number of aggregations to obtain the convex hull in \Cref{thm:opencase,thm:closedcase}. In all our experiments, we have not discovered an example that satisfies the \PDLC{} condition and for which the convex hull is obtained using an infinite intersection.

However, if we drop some of these conditions, namely \PDLC{} and the $n \geq 3$ condition, we can have that the convex hull of three quadratic inequalities is the intersection of an \emph{infinite} number of aggregated constraints. In this case, however, the aggregations provided by $\lambda$ may not satisfy $\nu(\genericquadopen_{\lambda}) \leq 1$. 

\begin{restatable}{myproposition}{infiniteaggregations}\label{example:4}
Consider the set 
$$\genericquadopen:= \{ x \in \mathbb{R}^2:  x_1^2 \leq 1,\  x_2^2 \leq 1, \ (x_ 1 - 1)^2 + (x_2 - 1)^2 \geq 1\},$$
and let $\Omega^+ := \{ \lambda \in \mathbb{R}^3_{+}: \genericquadopen_{\lambda} \supseteq \textup{conv}(\genericquadopen)\}$. It holds that
\[\conv(\genericquadopen) = \bigcap_{\lambda \in \Omega^+} \genericquadopen_{\lambda}.\] 
Moreover, \(\conv(\genericquadopen) \subsetneq \bigcap_{\lambda \in \tilde{\Omega}^+} \genericquadopen_{\lambda}\) for any $\tilde{\Omega}^+ \subseteq \Omega^+$ which is finite.
%
\end{restatable}
\Cref{fig:infiniteaggregation} illustrates the set considered in \Cref{example:4}, and the corresponding proof is provided in \Cref{sec:infiniteaggproof}.

\begin{figure}
    \centering
    \includegraphics[scale=0.25]{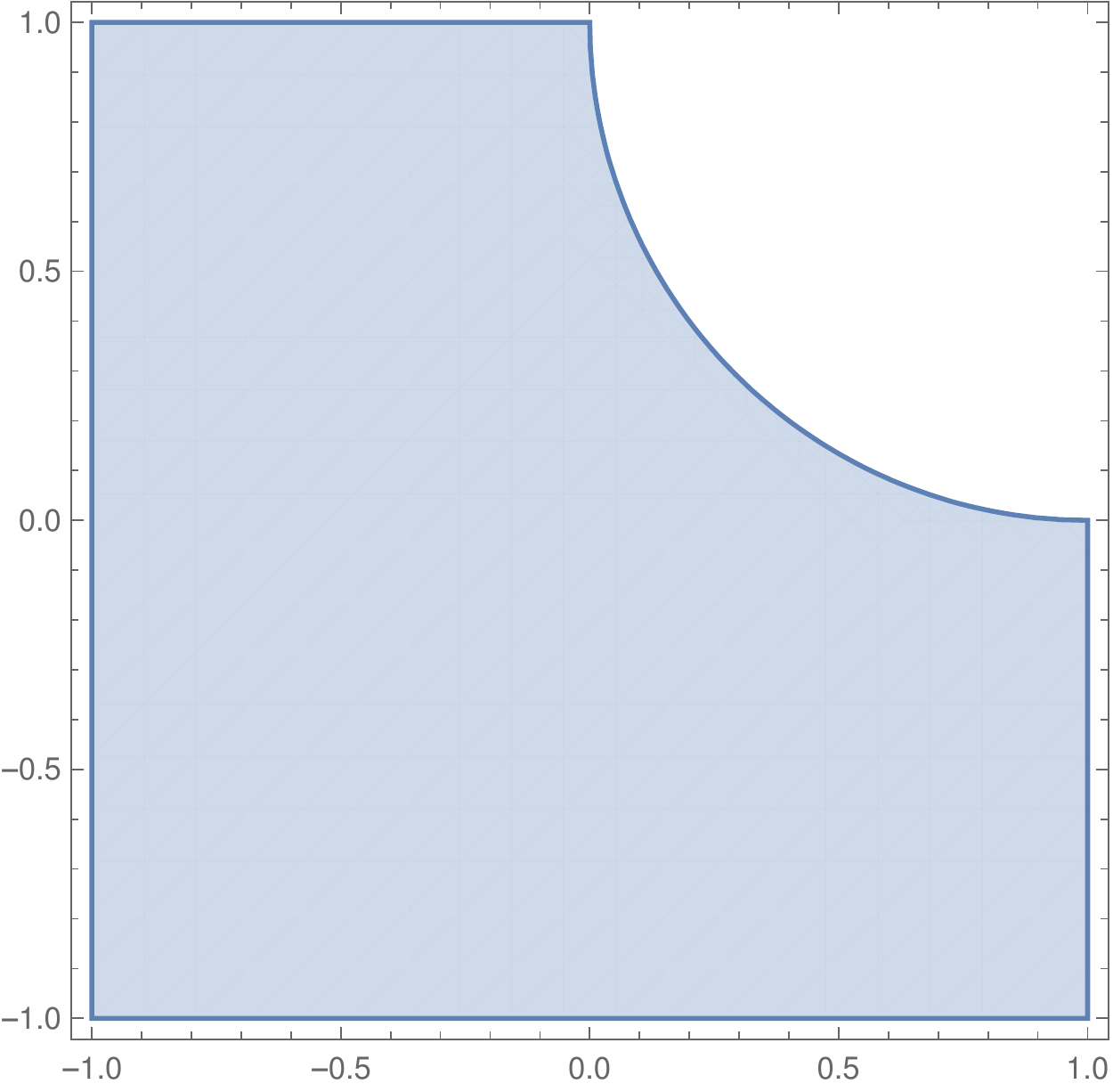}\hspace{.4cm}
    \includegraphics[scale=0.25]{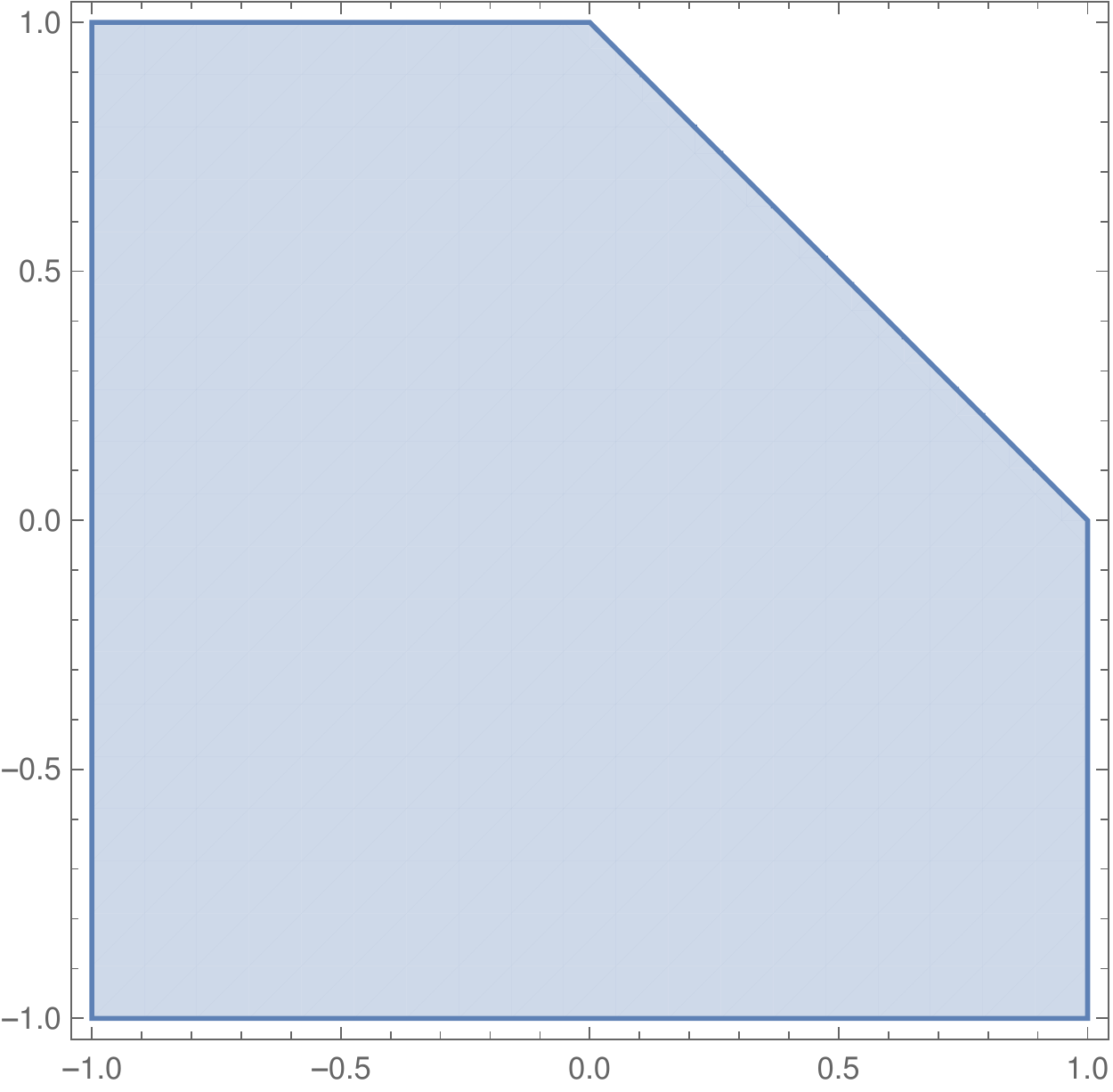}
    \caption{Plots of sets $\genericquadopen$ (left) and $\conv(\genericquadopen)$ (right) for \Cref{example:4}. In this example, the convex hull is obtained via an infinite intersection of aggregations of the three quadratic inequalities describing $\genericquadopen$.}
    \label{fig:infiniteaggregation}
\end{figure}

%


\section{Conclusion and open questions}\label{sec:conclusion}

Our results show that the aggregation approach remains expressive in the case of three quadratic inequalities, in the sense that aggregations can provide the convex hull of a set. This result requires the \PDLC{} condition, which seems very restrictive, but that cannot be completely avoided. We have also shown that it is unlikely to generalize these results to four or more constraints unless further structure is considered.

Our main open question involves the finiteness of the $\Omega$ set in \Cref{thm:opencase,thm:closedcase}: we have not been able to show whether an infinite number of aggregations is actually needed. We show an example where an infinite number of aggregations is required but without the \PDLC{} condition. We have not been able to produce such an example when the \PDLC{} condition holds.

We also leave two more open questions related to our results. Firstly, relaxing/replacing the \emph{No low-dimensional components} condition. This condition in \Cref{thm:closedcase} is not easy to verify, unfortunately. Replacing it or relaxing it would be a great way to strengthen this result.
Secondly, relaxing the \PDLC{} condition. Although we know that if we completely relax this condition, \Cref{thm:opencase} may not hold, we also see there are examples, such as \Cref{example:4}, where the convex hull is obtained via aggregation without the \PDLC{} condition holding. It would be interesting to better understand the role of this condition, although this direction seems 
challenging to pursue.

\section{Symmetric S-lemma for three quadratic constraints}\label{sec:proofslemma}

\subsection{Preliminary results} \label{sec:prelimopt} In this section we rely on two rank reduction results.
The following result is from~\cite{pataki1998rank}.
\begin{theorem}\label{thm:Pataki} If $\mathcal{A} \subseteq \mathbb{S}^n$ is an affine subspace such that the intersection $\mathbb{S}^n_+ \cap \mathcal{A}$ is non-empty and $\textup{dim}(\mathcal{A}) \geq {n+1 \choose 2 } - {r+2 \choose 2} +1$, then there is a matrix $X \in \mathbb{S}^n_+ \cap \mathcal{A}$ such that $\textup{rank}(X) \leq r$.
\end{theorem}

The following result is from~\cite{barvinok2001remark}.
\begin{theorem}\label{thm:Barvinok} If $\mathcal{A} \subseteq \mathbb{S}^n$ is an affine subspace such that the intersection $\mathbb{S}^n_+ \cap \mathcal{A}$ is non-empty, bounded and $\textup{dim}(\mathcal{A}) \geq {n+1\choose2} - {r+2\choose 2}$ then there is a matrix $X \in \mathbb{S}^n_+ \cap \mathcal{A}$ such that $\textup{rank}(X) \leq r$.
\end{theorem}

%


\subsection{Proof of \Cref{thm:Slemma3}}
To aid the reader, we restate the lemma here.
\Slemmathree*
\begin{proof}
The $\Leftarrow$ implication is trivial: by contradiction, if there is a $\hat{x}$ such that $g_i(\hat{x}) < 0$, then
\[\hat{x}^\top \left(\sum_{i=1}^3 \lambda_i Q_i \right) \hat x = \sum_{i=1}^3 \lambda_i g_i(\hat x) < 0,\]
a contradiction with $\sum_{i=1}^3 \lambda_i Q_i  \succeq 0$.

We now show the contrapositive of the other direction. Therefore, let us assume 
\begin{equation}\label{contradiction-assumption}
\nexists \lambda \in \mathbb{R}^3_+ \setminus \{0\},\, \textup{s.t.} \sum_{i=1}^3 \lambda_i Q_i \succeq 0
\end{equation}
and we show next that $\{g_i(x) < 0, i=1,2,3\} \neq \emptyset$. 
First of all, note that we can assume
\begin{equation}\label{2quad-assumption}
\{x\in \mathbb{R}^n \st g_i(x) < 0, i=2,3\} \neq \emptyset
\end{equation}
Since otherwise, by the S-lemma for 2 quadratics (\Cref{thm:Slemma2}), there exists $(\lambda_2, \lambda_3)\in \mathbb{R}^2_+ \setminus \{0\}$ such that
\[\lambda_2 Q_2 + \lambda_3 Q_3 \succeq 0.\]
Setting $\lambda_1 = 0$ contradicts \eqref{contradiction-assumption}. Now, consider the following SDP
\begin{subequations}\label{SDP-primal}
\begin{align}
\min\quad & \langle Q_1, X \rangle \\
\text{s.t.} \quad &  \langle Q_2, X \rangle \leq -1\\
& \langle Q_3, X \rangle  \leq -1 \\
& X \succeq 0
\end{align}
\end{subequations}
and its dual
\begin{subequations}\label{SDP-dual}
\begin{align}
\max\quad & -y_2 - y_3 \\
\text{s.t.} \quad &  y_2 Q_2 + y_3 Q_3 \preceq Q_1 \\
& y_2, y_3 \leq 0.
\end{align}
\end{subequations}
By \eqref{contradiction-assumption}, we obtain that \eqref{SDP-dual} is infeasible. Additionally, note that by \eqref{2quad-assumption},  problem \eqref{SDP-primal} satisfies Slater's conditions: indeed, for $\hat x$ in \eqref{2quad-assumption} the matrix
\[M\hat{x} \hat{x}^\top + \epsilon I \succ 0, \]
and is feasible for \eqref{SDP-primal}, for sufficiently large $M>0$ and small $\epsilon > 0$. SDP duality thus implies that  $\eqref{SDP-primal}$ must be unbounded. In particular, there exists $\hat X\succeq 0$ such that
\begin{align*}
\langle \hat X, Q_i \rangle = \epsilon_i && i=1,2,3,
\end{align*} 
for $\epsilon_i < 0$. This is almost what we want: we would like for $\hat X$ to have rank 1. We now work towards a rank reduction of $\hat X$.

Define the affine subspace
\[\mathcal{A}  = \{X \in \mathbb{S}^n \, :\, \langle X, Q_i \rangle = \epsilon_i,\, i=1,2,3\},\]
and $\mathcal{G} = \mathcal{A} \cap \mathbb{S}^n_+$. 

We claim that $\mathcal{G}$ is non-empty, bounded, and $\dim(\mathcal{A}) \geq {n+1\choose2} - 3$. Clearly $\hat{X}\in \mathcal{G}$.  Additionally, since there is a PD linear combination of $Q_1,Q_2,Q_3$, there exists $\theta_i$ such that
\[\sum_{i=1}^3 \theta_i Q_i \succ 0.\]
Note that
\[\mathcal{G} \subseteq \left\{X \succeq 0\, :\, \left\langle X, \sum_{i=1}^3 \theta_i Q_i \right\rangle \leq \sum_{i=1}^3 \theta \epsilon_i \right\}. \]
Since $\sum_{i=1}^3 \theta_i Q_i \succ 0$, the set on the right-hand side is bounded and thus $\mathcal{G}$ is bounded as well.

The statement regarding the dimension of $\mathcal{A}$ follows since $\mathbb{S}^n$ has dimension ${n+1\choose2}$ and $\mathcal{A}$ is defined via 3 affine equality constraints. We construct a rank 1 element of $\mathcal{G}$ distinguishing the following cases.

\paragraph{\bf Case 1: $\dim(\mathcal{G}) = {n+1\choose2} - 3$} In this case 
\[ \dim(\mathcal{G}) = {n+1\choose2} - {r+2\choose 2} \]
for $r=1$. By \Cref{thm:Barvinok} (see \cite{barvinok2001remark}), there exists $\tilde{X}\in \mathcal{G}$ such that
\[\text{rank}(\tilde X) \leq r = 1.\]

\paragraph{\bf Case 2: $\dim(\mathcal{G}) > {n+1\choose2} - 3$} In this case 
\[\dim(\mathcal{G}) \geq {n+1\choose2} - {r+2\choose 2} +1 \]
for $r=1$. By \Cref{thm:Pataki} (see \cite{pataki1998rank}), there exists $\tilde{X}\in \mathcal{G}$ such that
\[\text{rank}(\tilde X) \leq r = 1.\]

In both cases, we conclude that there exists $\tilde{x}\in \mathbb{R}^n$ such that
\[\tilde{x}^\top Q_i \tilde{x} = \epsilon_i < 0, \quad i=1,2,3.\]
\end{proof}

\section{Convex hull via aggregations: open case}\label{sec:proofthmopencase}

\subsection{Preliminary results} We begin by presenting the main tools we use in this section.
The following theorem is a classical linear algebra result. A proof may be found in ~\cite{Horn1985matrix}. 

\begin{theorem}[Cauchy's Interlacing theorem]\label{thm:Cauchy} Let $k,n$ be integers satisfying $1 \leq k < n$. Let $A \in \mathbb{S}^n$ and denote its eigenvalues $\lambda_1(A) \geq \lambda_1(A)  \geq \dots \geq \lambda_n(A)$. 
Let $\tilde{A}$ be any of its $k\times k$ principal submatrices. Then: 
$$ \lambda_{n-k+i}(A)\leq \lambda_i (\tilde{A})\leq \lambda_i (A). 	$$
\end{theorem}

We heavily rely on the next definition from \cite{yildiran2009convex}.
\begin{definition}
A semi-convex cone (SCC) is the union of two convex cones which are symmetric reflections of each other with respect to the origin. 
\end{definition}
The following result is from~\cite{yildiran2009convex}. 
\begin{theorem}\label{thm:SCC}
Let $\mathcal{P} =\{ x\in \mathbb{R}^{n+1}\, : \, x^{\top} P x < 0\} \neq \emptyset$. The following are equivalent:
\begin{enumerate}
\item There exists a linear hyperplane that does not intersect $\mathcal{P}$.
\item $P$ has one negative eigenvalue.
\item $\mathcal{P}$ is a semi-convex cone.
\end{enumerate}
\end{theorem}

Henceforth, we refer to a linear hyperplane that does not intersect $\mathcal{P}$, an SCC of the form $\{x: x^{\top}P x < 0\}$, as a \emph{faux-separating} hyperplane of $\mathcal{P}$. The paper~\cite{yildiran2009convex} uses the term ``separating'', however, we prefer to reserve the term ``separating'' to refer to a separating hyperplane with respect to a set. In this case, a hyperplane that is \emph{faux}-separating $\mathcal{P}$ is splitting the set in two and both the half-spaces corresponding to the faux-separating hyperplane have a non-empty intersection with $\mathcal{P}$. In particular, $\mathcal{P}$ intersected with each half-space is a convex cone. In \Cref{fig:fauxseparating} we illustrate a faux-separating hyperplane.

\begin{figure}
    \centering
    \includegraphics[scale=0.2]{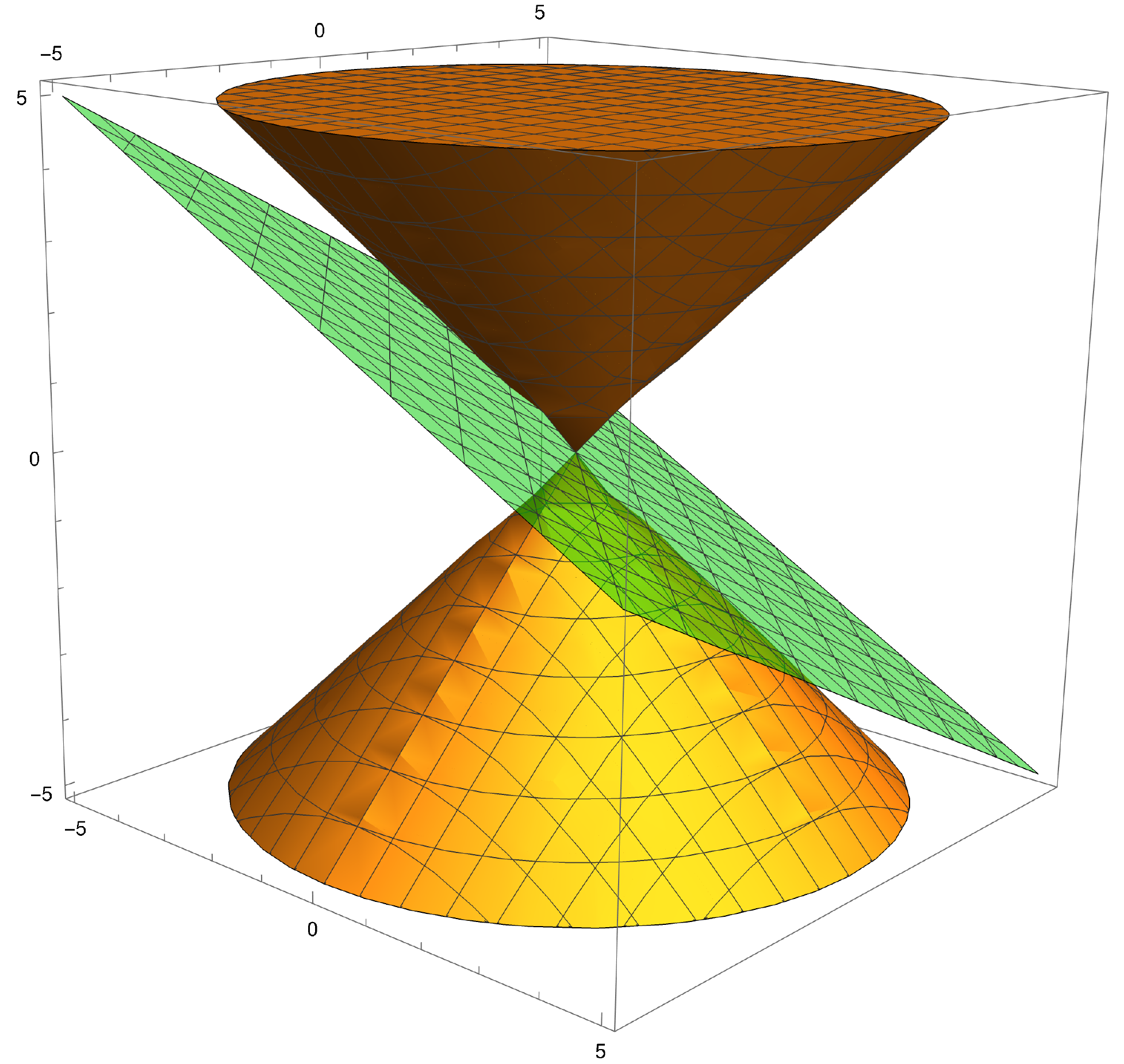}
    \caption{Example of a semi-convex cone $\mathcal{P}$ (orange) and a faux-separating hyperplane of $\mathcal{P}$ (green).}
    \label{fig:fauxseparating}
\end{figure}

\subsection{Separation in original space to homogenized space}\label{sec:proofopencase1}
Consider an arbitrary quadratic set 
\[ \genericquadopen = \left\{x: [x\quad 1] \begin{bmatrix} A_i & b_i \\ b^{\top}_i & c_i  \end{bmatrix} \begin{bmatrix} x \\ 1 \end{bmatrix} < 0, \ i \in [m] \right\}\]
and its homogenized version
\[ \genericquadopen^h = \left\{x:  [x\quad x_{n+1}] \begin{bmatrix} A_i & b_i \\ b^{\top}_i & c_i  \end{bmatrix} \begin{bmatrix} x \\ x_{n+1} \end{bmatrix}< 0, \ i \in [m] \right\}.\]

\begin{mylemma}\label{prop:validineqseparates}
Let $\alpha^{\top}x < \beta$ be a valid inequality for $\conv(\genericquadopen)$. If $\textup{conv}(\genericquadopen) \neq \mathbb{R}^n$, then $\{x: \alpha^\T x = \beta x_{n+1}\} \cap \genericquadopen^h = \emptyset$.
\end{mylemma}
\begin{proof}
By contradiction, suppose there is a $(\hat{x},\hat{x}_{n+1}) \in  \{x : \alpha^\T x = \beta x_{n+1}\} \cap S^h $.  Then
\[\hat{x}^\T A_i \hat{x} + 2b_i^\T \hat{x}\hat{x}_{n+1} + c_i \hat{x}_{n+1}^2 < 0\quad \textup{for all }\ i \in [m].\]
If $\hat{x}_{n+1} \neq 0$ this implies that $\frac{1}{\hat{x}_{n+1}} \hat{x} \in S$ and $\alpha^\T \hat{x}/\hat{x}_{n+1} = \beta$. This is a contradiction with $\alpha^{\top}x < \beta$ being valid for $\conv(\genericquadopen)$.

If $\hat{x}_{n+1} = 0$, we obtain $\hat{x}^\T A_i \hat{x} < 0$ for $i \in [m]$. Consider $z\in \mathbb{R}^n$ arbitrary and $z^{\pm} = z \pm M \hat{x}$, then
\[(z^{\pm})^\top A_i   z^{\pm} + 2b_i^\top z^{\pm} + c_i = M^2 (\hat{x}^\T A_i \hat{x}) + M \tilde{b}_i + \tilde{c}_i \]
for some scalars $\tilde{b}_i , \tilde{c}_i$ that do not depend on $M$. This implies that for $M$ large enough $z^\pm \in S$, and thus $z\in \conv(\genericquadopen)$. This contradicts $\conv(\genericquadopen) \neq \mathbb{R}^n$.
\end{proof}

\subsection{Towards faux-separation of an aggregated homogeneous constraint}\label{sec:proofopencase2}
Let
\[\mathcal{Q} = \{x\in \mathbb{R}^{n+1}\, : \, g_i(x) < 0,\, i=1,2,3\}, \]
where $g_i(x) = x^\T Q_i x$, $n\geq 3$ and, 
\[\mathcal{Q}_\lambda = \{x\in \mathbb{R}^{n+1}\, : \, \sum_{i=1}^3 \lambda_i g_i(x) < 0\}. \]
Relying on \Cref{thm:Slemma3}, we can obtain the following result. 
\begin{mylemma}\label{thm:separation} Let $n \geq 3$ and let $H\subseteq \mathbb{R}^{n+1}$ be a linear hyperplane. Suppose that $\{Q_i\}_{i=1}^3 \subset \mathbb{S}^{n+1}$ satisfy the \PDLC{} condition. Then $H\cap \mathcal{Q} = \emptyset$ if and only if $\exists \lambda \geq 0$ such that $H\cap \mathcal{Q}_\lambda = \emptyset$.
\end{mylemma}
\begin{proof}
The ``if'' part is direct: suppose $\exists \lambda \geq 0$ such that $H\cap \mathcal{Q}_\lambda=\emptyset$. Since $\mathcal{Q}_\lambda \supseteq \mathcal{Q}$, we immediately conclude $H\cap \mathcal{Q} = \emptyset$.

Now suppose $H\cap \mathcal{Q} = \emptyset$. We can parameterize the hyperplane $H$ as 
\[H = \{x\in \mathbb{R}^{n+1}\, : \, x = Uw, \, w\in \mathbb{R}^n\}. \]
where the $n$ columns of $U$ are linearly independent.
Therefore, the system
\begin{align}
w^\T U^\T Q_i U w <\, & 0 \quad i=1,2,3  \label{eq:infeasible}
\end{align}
is infeasible. Note that assuming \PDLC{} for $Q_i$ implies \PDLC{} for $U^\T Q_i U$. Indeed, let $\theta_i$ be the multipliers for the former, then
\[w^\T \left(\sum_{i=1}^3 \theta_i U^\T Q_i U \right) w = w^\T U^\T \left(\sum_{i=1}^3 \theta_i Q_i\right) U w.\]
Since $Uw = 0 \Rightarrow w = 0$, we conclude this is a PD linear combination of $U^\T Q_i U$. Thanks to this, we can apply \Cref{thm:Slemma3} and show that the infeasibility of \eqref{eq:infeasible} implies that there exists $\lambda\geq 0$ such that
\[ \sum_{i=1}^3 \lambda_i U^\T Q_i U \succeq 0.\]
Thus, there is no $w\in \mathbb{R}^n$ such that
\[ \sum_{i=1}^3 \lambda_i w^\T U^\T Q_i Uw  < 0\]
i.e., there is no $x\in H$ such that
\[ \sum_{i=1}^3 \lambda_i x^\T Q_i x  < 0\]
This implies $H\cap \mathcal{Q}_\lambda =\emptyset$.
\end{proof}

\subsection{Proof of \Cref{thm:opencase}}\label{sec:puttingittogetheropencase}
%
%
We are now ready to prove \Cref{thm:opencase}. For the convenience of the reader, we restate it here.
\thmopen*
\begin{proof}
The $\subseteq$ inclusion follows by definition. For the other direction, it suffices to take $\hat{x}\not\in \conv(\genericquadopen)$, and show that there exists $ \lambda\in \Omega$ such that $\hat{x} \not\in \genericquadopen_\lambda$. 

Since $\hat{x}\not\in \conv(\genericquadopen)$, and $\conv(\genericquadopen)$ is open, there exists $\alpha\in\mathbb{R}^n\setminus \{0\}$, such that
\[\conv(\genericquadopen) \subseteq \{x\st \alpha^\top x < \alpha^\top \hat{x}\}. \]
%
%
Since $\conv(\genericquadopen) \neq \mathbb{R}^n$, by \Cref{prop:validineqseparates} we have that $\{ (x, x_{n +1}) \in \mathbb{R}^{n +1}: \alpha^\top{x} = (\alpha^{\top}\hat{x}) x_{n +1}\}\cap \genericquadopen^h = \emptyset.$ Now, by applying \Cref{thm:separation}, with $H = \{ (x, x_{n +1} ) \in \mathbb{R}^{n +1}: \alpha^\top{x} = (\alpha^{\top}\hat{x}) x_{n +1}\}$ we obtain that there exists $\lambda\in \mathbb{R}^3_+$ such that
\begin{equation}\label{eq:inhomspace}
 (\genericquadopen^h)_\lambda \cap H = \emptyset.
\end{equation}
%
Note that $(\hat{x},1)\in H$. This implies $(\hat{x},1)\not\in (\genericquadopen^h)_\lambda$, i.e.,
\[ [\hat{x}\quad 1] \left(\sum_{i=1}^3 \lambda_i \begin{bmatrix} A_i & b_i \\ b^{\top}_i & c_i  \end{bmatrix} \right) \begin{bmatrix} \hat{x} \\ 1 \end{bmatrix} \geq 0,\]
therefore, $\hat{x}\not\in \genericquadopen_\lambda$. It remains to argue that $\lambda \in \Omega$. 

From \Cref{thm:SCC}, we know that $(\genericquadopen^h)_\lambda$ is an SCC, and $H$ faux-separates it. On one hand, this settles that $\nu(\genericquadopen^h_\lambda) = 1$ implying that $\nu(\genericquadopen_\lambda) \leq 1$ (by the interlacing \Cref{thm:Cauchy}).
Additionally,
\[(\genericquadopen^h)_\lambda^+ \coloneqq (S^h)_\lambda \cap \{(x, x_{n +1})\st \alpha^\top x < (\alpha^{\top}\hat{x}) x_{n+1}\} \]
is a convex set.  We claim that 
\[\conv(\genericquadopen) \times \{1\} \subseteq (\genericquadopen^h)_\lambda^+. \]
Indeed, every $\tilde{x} \in S$ satisfies $\alpha^\top \tilde{x} < \alpha^{\top}\hat{x}$, and thus
\[(\tilde{x},1) \in (\genericquadopen^h)_\lambda^+.\]
which implies
\[\genericquadopen\times \{1\} \subseteq (\genericquadopen^h)_\lambda^+.\]
The claim follows as $(\genericquadopen^h)_\lambda^+$ is convex.
To conclude, we note that this implies
\begin{align*}
\conv(\genericquadopen) \times \{1\} & \subseteq (\genericquadopen^h)_\lambda^+ \cap \{(x,x_{n+1}) \st x_{n+1} = 1\} \\
& \subseteq (\genericquadopen^h)_\lambda \cap \{(x,x_{n+1}) \st x_{n+1} = 1\} \\
& = \genericquadopen_\lambda \times \{1\},
\end{align*}
which completes the proof that $\lambda \in \Omega$.
\end{proof}

\section{Convex hull via aggregations: closed case}\label{sec:proofclosedcase}

%
\subsection{Preliminary results}
In this section, we use the following result from~\cite{Modaresi2017}.
\begin{mylemma}[\cite{Modaresi2017}]\label{lemma:jp}
Let $A$ and $B$ be two non-empty closed sets such that
\[A \subseteq \overline{\inte(A)}\]
and $B$ is convex. If $\conv(\inte(A)) = \inte(B)$, then $\clconv(A) = B$.
\end{mylemma}

Additionally, we need to prove the following lemma.
\begin{mylemma}\label{prop:interiors}
If $\{A_i\}_{i\in I}$ is a collection of sets such that $\bigcap_{i\in I} \inte(A_i)$ is open, then
\[\inte\left(\bigcap_{i\in I} A_i \right) = \bigcap_{i\in I} \inte(A_i).\] 
\end{mylemma}
Before moving to the proof, we note that in \cite{Modaresi2017} ---where the authors obtain a result as Yildiran's for the closed case--- such a lemma is not needed. This is because they deal with a finite intersection, and the interior behaves well with a finite intersection. In our case, we may have an infinite intersection.
\begin{proof}
The inclusion $\subseteq$ is always true, therefore we only need to show $\supseteq$. Since $\inte(A_i)\subseteq A_i$, it always holds that
\[\bigcap_{i\in I} \inte(A_i) \subseteq \bigcap_{i\in I} A_i\]
Since the set on the left is open, and the interior of a set is the largest open set contained in it, we conclude
\[\bigcap_{i\in I} \inte(A_i) \subseteq \inte\left(\bigcap_{i\in I} A_i \right) \]
\end{proof}

\subsection{Proof of \Cref{thm:closedcase}}

Using the results from the previous section, we are ready to prove \Cref{thm:closedcase} which, for the convenience of the reader, we restate here.
\thmclosed*
\begin{proof}
We know from \Cref{thm:opencase} that
\[\conv(\inte(\genericquadclosed)) = \bigcap_{\lambda\in \tilde{\Omega}} \inte(\genericquadclosed_\lambda)\]
where 
\[\tilde{\Omega} \coloneqq \{\lambda \in \mathbb{R}^3_+\st \inte(\genericquadclosed_\lambda) \supseteq \conv(\inte(\genericquadclosed)) \textup{ and } \nu(\genericquadclosed_\lambda) \leq 1 \}. \]
Since $\conv(\inte(\genericquadclosed))$ is an open set,  the intersection $\bigcap_{\lambda\in  \tilde{\Omega}} \inte(\genericquadclosed_\lambda)$ is open as well. \Cref{prop:interiors} then implies
\[\bigcap_{\lambda\in \tilde{\Omega}} \inte(\genericquadclosed_\lambda) = \inte\left(\bigcap_{\lambda\in \tilde{\Omega}} \genericquadclosed_\lambda\right)\]
We would like to use $A = S$ and $B= \bigcap_{\lambda\in \tilde{\Omega}} \genericquadclosed_\lambda$  in \Cref{lemma:jp}. $A$ satisfies the necessary hypothesis directly. $B$ is closed since it is the intersection of closed sets, and it is convex since its interior is convex. Therefore, \Cref{lemma:jp} implies that
\[\clconv(\genericquadclosed) = \bigcap_{\lambda\in \tilde{\Omega}} \genericquadclosed_\lambda.\]
It remains to show that we can replace $\tilde{\Omega}$ with $\Omega'$ in the intersection. Recall that, by definition of $\Omega'$ we have
\[\conv(\genericquadclosed) \subseteq \bigcap_{\lambda\in \Omega'} \genericquadclosed_\lambda\]
and since the set on the right is closed we have
\[\clconv(\genericquadclosed) \subseteq \bigcap_{\lambda\in \Omega'} \genericquadclosed_\lambda.\]
For the other direction, we first prove that $\tilde{\Omega} \subseteq \Omega'$. Indeed, the condition regarding the negative eigenvalue remains unchanged, and under the assumption $\genericquadclosed \subseteq \overline{\inte(\genericquadclosed)}$ we have $\clconv(\inte(\genericquadclosed)) = \clconv(\genericquadclosed)$. Thus,
\begin{align*}
\lambda \in \tilde{\Omega}\Rightarrow \inte(\genericquadclosed_\lambda) \supseteq \conv(\inte(\genericquadclosed)) & \Rightarrow \genericquadclosed_\lambda \supseteq \clconv(\inte(\genericquadclosed)) \\ 
& \Leftrightarrow \genericquadclosed_\lambda \supseteq \clconv(\genericquadclosed) \\ 
& \Rightarrow \genericquadclosed_\lambda \supseteq \conv(\genericquadclosed) \\
& \Rightarrow \lambda \in \Omega'. \\
\end{align*}
Therefore
\[\clconv(\genericquadclosed) = \bigcap_{\lambda\in \tilde{\Omega}} \genericquadclosed_\lambda \supseteq \bigcap_{\lambda\in \Omega'} \genericquadclosed_\lambda. \]
\end{proof}

\section{Counterexample proofs}\label{sec:counterexproofs}

\subsection{Proof of \Cref{prop:4quadexample}} \label{sec:4quadproof}
This example shows that a generalization of our main theorem is unlikely to hold in the case of four quadratics.
\fourquadexample*
The rest of this subsection is dedicated to the proof of this proposition.

First of all, note that $\sum_{i=1}^4 \theta_i \begin{bmatrix} A_i & b_i \\ b_i & c_i  \end{bmatrix} \succ 0$ where $\theta_1 = -1$, $\theta_2 = -40$, $\theta_3 = -40$, $\theta_4 = -40$, i.e. the \PDLC{} condition holds.  Also it is easily verified that $\textup{conv}(\genericquadopen) \neq \mathbb{R}^3$ (we actually show below that $S$ is bounded). However, in this example 
\[\conv(S) \subsetneq \bigcap \{ \genericquadopen_\lambda \st \lambda\in \mathbb{R}^n_+,\, \conv(S) \supseteq \genericquadopen_\lambda \}.\]
To show this last claim, we begin by showing that the set is bounded. While this can be verified numerically (see \Cref{fig:counterex4quads}) using a general-purpose solver, the non-convex nature of the set can impair solvers in certifying global optimality. For this reason, we provide a more analytical proof sketch of this claim which involves exact calculations. We do not provide all details, since they would result in a tedious proof, but we provide enough detail to support our claims.

It is worth mentioning that boundedness cannot be obtained from a simple aggregation-based argument: in \cite{polik2007survey}, the authors show that there is no $(\lambda_2,\lambda_3,\lambda_4) \geq 0$ such that
\[A_1 + \lambda_2 A_2 + \lambda_3 A_3 + \lambda_4 A_4\succ 0.\]
The same argument can be adapted to show that there is no $(\lambda_1, \lambda_2,\lambda_3,\lambda_4) \geq 0$ such that
\[\lambda_1 A_1 + \lambda_2 A_2 + \lambda_3 A_3 + \lambda_4 A_4\succ 0.\]
The standard SDP relaxation, which could also be considered to show boundedness analytically, it is not useful either, as it is unbounded.
\begin{claim}\label{claim:boundedS}
The set $S$ is bounded. Moreover, it is contained in the box $[-8,8]^3$.
\end{claim}
We remark that the bounding box $[-8,8]^3$ is by no means optimal. Our numerical experiments indicate that it is actually $[-1.23,1.23]^3$. We just focus on a box whose validity proof is simple and rigorous.
\begin{proof}(\Cref{claim:boundedS})
We begin by noting that $A_i$ $i=2,3,4$ induce the following system:
\begin{subequations}\label{eq:system2-4}
\begin{align}
    -2.1 x_1^2 + x_2^2 + x_3^2 <& 0 \\
    x_1^2 - 2.1 x_2^2 + x_3^2 <& 0 \\
    x_1^2 + x_2^2 -2.1 x_3^2 <& 0 
\end{align}
\end{subequations}
If we add and subtract $x_1^2, x_2^2$, and $x_3^2$ in the first, second, and third inequality, respectively, we obtain
\begin{equation}
    \label{eq:relaxationwithmin}
    x_1^2 + x_2^2 + x_3^2 < 3.1\min\{x_1^2,x_2^2, x_3^3\}
\end{equation}
Note that 
\[\min\{x_1^2,x_2^2, x_3^3\} \leq \min\{|x_1 x_2|,|x_2 x_3|, |x_1 x_3|\}.\]
Using that the minimum is bounded by the average, we obtain that the following inequality defines a relaxation of \eqref{eq:system2-4}
\begin{equation}
    \label{eq:pwconvexrelaxation}
    x_1^2 + x_2^2 + x_3^2 < \frac{31}{30}\left(|x_1 x_2|+|x_2 x_3|+ |x_1 x_3|\right)
\end{equation}
Consider now the aggregation of $x^\top A_1 x - 1 < 0$ and \eqref{eq:pwconvexrelaxation} with weights $\frac{1}{2}$ and 1, respectively. We obtain the following aggregated constraint:
\begin{align*}
0>&    \frac{1}{2}\left( x_1^2 + x_2^2 + x_3^2  + 2.2 x_1 x_2 + 2.2 x_2x_3 + 2.2 x_1x_2 - 1\right)\\
& + x_1^2 + x_2^2 + x_3^2 - \frac{31}{30}\left(|x_1 x_2|+|x_2 x_3|+ |x_1 x_3|\right)
\end{align*}
Let us now restrict to an orthant, that is, we fix the sign of each $x_i$. In this case, it is not hard to see that the last expression becomes
\[x^\top Q x < \frac{1}{2}\]
with
\[\displaystyle Q = 
\begin{bmatrix} 
\frac{3}{2} & \frac{11}{20} - \frac{31}{60}q_{1,2} & \frac{11}{20} - \frac{31}{60} q_{1,3} \\
 \frac{11}{20} - \frac{31}{60} q_{1,2} & \frac{3}{2} & \frac{11}{20} - \frac{31}{60} q_{2,3} \\
  \frac{11}{20} - \frac{31}{60} q_{1,3} & \frac{11}{20} - \frac{31}{60} q_{2,3} & \frac{3}{2} 
\end{bmatrix}\]
where each $q_{i,j} = \text{sgn}(x_ix_j) \in \{-1,1\}$ is fixed in any orthant. A case analysis on $Q$ shows that it is always a positive definite matrix. Indeed, the case when $\lambda_{\min}(Q)$ is the smallest is obtained when exactly two among $\{q_{1,2}, q_{2,3}, q_{1,3}\}$ take the value $-1$. In this case, the smallest eigenvalue is
\[\frac{22}{15 \left(\sqrt{8193}+91\right)}  > \frac{8}{1000}.\]
This immediately shows boundedness of $S$. To obtain a bounding box, we use that an upper bound on the length of a vector of an ellipsoid is related to the smallest eigenvalue of the matrix defining it:
\[\max\{ \|x\| \st x^\top Q x \leq 1/2 \} = \frac{1}{\sqrt{\lambda_{\min} (2Q)}} \leq \sqrt{\frac{1000}{16}} < 8\]
In particular, this shows that $[-8,8]^3$ is a valid bounding box for $S$.

\end{proof}

\begin{claim} \label{claim:threepoints}
\begin{align*}
%
%
x^1 = \left(\frac{1207}{1000}, -\frac{117}{4000}, -\frac{117}{4000} \right) \in& \conv(S) \\
x^2 = (10, -5, -5)\not\in& \conv(S) 
\end{align*}
\end{claim}
\begin{proof}
The fact that $x^2 \not\in \conv(S)$ follows directly from \Cref{claim:boundedS}, as it lies outside the bounding box. Consider
\begin{align*}
\tilde{x}^1 &= \left(\frac{1207}{1000}, \frac{1207}{1000}, -\frac{2531}{2000} \right) \\
\tilde{x}^2 &= \left(\frac{1207}{1000}, -\frac{2531}{2000} , \frac{1207}{1000}\right)
\end{align*}
One can easily verify that $\tilde{x}^1\in S$ $i=1,2$ simply evaluating the corresponding quadratic inequalities.
The result follows from noting that
\begin{align*}
x^1 =& \frac{1}{2}(\tilde{x}^1 + \tilde{x}^2) 
\end{align*}

\end{proof}

\begin{claim}
There is no $\lambda \geq 0$ such that
\begin{equation}
x^1 \in S_\lambda \textup{ and } x^2 \not\in S_\lambda. \label{eq:conditionforlambda}
\end{equation}
\end{claim}
\begin{proof}
Let us call
\[q_i(x) \coloneqq [x\quad 1] \begin{bmatrix} A_i & 0 \\0 & c_i  \end{bmatrix} \begin{bmatrix} x \\ 1 \end{bmatrix} \qquad i\in [4] \]
A vector $\lambda \geq 0$ that satisfies \eqref{eq:conditionforlambda} also satisfies, without loss of generality, the following linear system
\begin{subequations}\label{eq:infeassystem}
\begin{align}
    \sum_{i=1}^4 \lambda_i q_i(x^1) & < 0 \\
    \sum_{i=1}^4 \lambda_i q_i(x^2) & > 0 \\
    \lambda_1 + \lambda_2 + \lambda_3 + \lambda_4 & = 1\\
    \lambda & \geq 0
\end{align}
\end{subequations}
However, it can be directly verified that the linear system \eqref{eq:infeassystem} is infeasible. We prefer to not provide the exact numerical values in \eqref{eq:infeassystem} as they possess too many digits. Nonetheless, the reader can verify this claim computationally using exact arithmetic based on the values provided for $x^1$ and $x^2$.

For the sake of completeness, we provide approximate coefficients for \eqref{eq:infeassystem}, along with its infeasibility proof. System \eqref{eq:infeassystem} has the form
\begin{subequations}\label{eq:infeassystem2}
\begin{align}
    0.3051\lambda_1 - 3.0576 \lambda_2 + 1.4559 \lambda_3 + 1.4559 \lambda_4  & < 0 \\
    -16\lambda_1 - 160 \lambda_2 + 72.5 \lambda_3 + 72.5 \lambda_4 & > 0 \\
    \lambda_1 + \lambda_2 + \lambda_3 + \lambda_4 & = 1\\
    \lambda & \geq 0
\end{align}
\end{subequations}
Aggregating the first three linear expression with  weights $[1.7629, -0.0342, -0.0854]$ yields the infeasibiliby proof
\[\lambda_1 < -0.08544. \]
\end{proof}

\subsection{Proof of \Cref{example:3}}\label{sec:3quadNotPDLC}
In this section, we show a case where relaxing the \PDLC{} condition results in a convex hull not obtainable via aggregations.
\threequadnotpdlc*
Let us show that $\conv(S)\neq \mathbb{R}^3$: indeed, one can easily prove that $\genericquadopen$ is bounded, as the constraints are simply
\begin{align*}
    x_1^2 & < 1 \\
    x_2^2 & < 1 \\
    x_3^2 & < x_1x_2.
\end{align*}
See \Cref{fig:counterexPDLC}.
Additionally, there is no linear combination of $\sum_{i = 1}^3 \theta_i \begin{bmatrix} A_i & b_i \\ b^{\top}_i & c_i  \end{bmatrix}$ which is positive definite: if there were, we would need $\theta_1 < 0$ or $\theta_2 < 0$, since $c_3 =0$ and $c_1 =c_2 =-1$. However, in that case the first or the second diagonal entry of the linear combination is negative.

Now we prove that \(\conv(S) \subsetneq \bigcap \{ \genericquadopen_\lambda \st \lambda\in \mathbb{R}^n_+,\, \conv(S) \supseteq \genericquadopen_\lambda \}.\)
\begin{claim}
$\left(-\frac{3}{8} - \epsilon, \frac{3}{8} +\epsilon, \frac{1}{2} \right) \not \in \conv(S)$ for any $\epsilon>0$. 
\end{claim}
\begin{proof}
Consider the halfspace $H: \{x \in \mathbb{R}^3: -x_1 + x_2 +x_3 \leq 1.25 \}$. Notice that $- \left(-\frac{3}{8} - \epsilon \right) + \left(\frac{3}{8} +\epsilon \right) + \frac{1}{2} = 1.25 +2\epsilon > 1.25$. We proceed to show that $\genericquadopen \subseteq H$, which completes the proof of the claim.

Let $\hat{S}(z):= \textup{conv}(S \cap \{x\in \mathbb{R}^3: x_3 = z\})$. It is straightforward to verify that $\hat{S}(z)$ is a polytope with extreme points: $(1, 1, z)$, $(z^2, 1, z)$, $(1, z^2, z)$, $(-1, -z^2, z)$, $(-z^2, -1, z)$ and $(-1, -1, z)$. Thus, 
$$\textup{max}\{-x_1 + x_2+ x_3: x\in \hat{S}(z)\} = \textup{max}\{z, -1 + z^2 + z, 1 +z - z^2\}.$$ 
Noting that $\textup{max}\{x^2_3: x\in S\} =1$, we have that 
\begin{eqnarray*}
\textup{sup}\{-x_1 + x_2+ x_3: x\in S\} &\leq& \textup{max}\{-x_1 + x_2+ x_3: x\in \cup_{z: |z| \leq 1}\hat{S}(z)\}\\
 &=& \textup{max}_{z: |z| \leq 1}\{z, -1 + z^2 + z, 1 +z - z^2\} = 1.25.
\end{eqnarray*}
Therefore, $\genericquadopen\subseteq H$.
\end{proof}
\begin{claim}
$(0, 0, 1- \epsilon) \in \textup{conv}(\genericquadopen)$ for $\epsilon\in (0,1)$.
\end{claim}
\begin{proof}
Indeed, it can be directly verified that $(1 - \epsilon/2, 1- \epsilon/2, 1 - \epsilon), (-1 + \epsilon/2, -1 + \epsilon/2, 1- \epsilon) \in S$. The midpoint of these two points is $(0, 0, 1- \epsilon)$, which shows it lies in $\textup{conv}(S)$.
\end{proof}

Based on the two claims above, it is sufficient to prove that for a particular $\epsilon\in (0,1)$ there is no $\lambda\in \mathbb{R}^3_+ \setminus \{0\}$, such that $\left(-\frac{3}{8} - \epsilon, \frac{3}{8} +\epsilon, \frac{1}{2} \right) \not \in \genericquadopen_{\lambda}$ and $(0, 0, 1- \epsilon) \in \genericquadopen_{\lambda}$. 

Let $\epsilon = 1/8$ and, by contradiction, assume there is $\lambda\in \mathbb{R}^3_+ \setminus \{0\}$ such that $x = \left(-\frac{3}{8} - \epsilon, \frac{3}{8} +\epsilon, \frac{1}{2} \right) \not \in \genericquadopen_{\lambda}$ and $y = (0, 0, 1- \epsilon) \in \genericquadopen_{\lambda}$. A simple calculation yields
\[x_1^2 - 1 = -\frac{3}{4}, \quad x_2^2 - 1 = -\frac{3}{4},\quad x_3^2 - x_1x_2 = \frac{1}{2}\]
\[y_1^2 - 1 = -1, \quad y_2^2 - 1 = -1,\quad y_3^2 - y_1y_2 = \frac{49}{64}\]
This implies that $\lambda$ satisfies
\begin{align*}
    -\frac{3}{4} \lambda_1 - \frac{3}{4} \lambda_2 + \frac{1}{2} \lambda_3 & \geq 0 \\
    - \lambda_1 -  \lambda_2 + \frac{49}{64} \lambda_3 & < 0
\end{align*}
Multiplying the first inequality by $-4/3$ and adding it to the second inequality yields
\[\left(-\frac{2}{3} + \frac{49}{64}\right) \lambda_3 < 0\]
This implies $\lambda_3 < 0$, which is a contradiction. We conclude that there is no such $\lambda$.



\subsection{On the SDP tightness of Example \ref{example:1-2}} \label{sec:SDPtightness}
In this example, we are considering a set $T$ defined by the following inequalities
\begin{subequations}\label{eq:Texample}
\begin{align}
    x_1^2 + x_2^2 &\leq 2\\
    -x_1^2 - x_2^2 & \leq -1 \\
    -x_1^2 + x_2^2 + x_3^2 + 6x_1 &\leq 0. \label{eq:Texample-third}
\end{align}
\end{subequations}
We would like to argue that $\max\{x_1\st x\in T\} < 0$, while its SDP bound is $>0$.

If we aggregate inequalities \eqref{eq:Texample} with multipliers $(0,1/2,1/2)$ we obtain the implied inequality
\[-x_1^2 + \frac{1}{2}x_3^2 + 3x_1 + \frac{1}{2} \leq 0.\]
Lower bounding $x_3^2 \geq 0$, and factoring the resulting concave quadratic yields
\[\left(x_1 - \frac{1}{2}(3-\sqrt{11}) \right) \left( - x_1 + \frac{1}{2}(3+\sqrt{11}) \right) \leq 0,\]
from where we conclude 
\begin{equation}
    \label{eq:twocases}
    x_1 \leq  \frac{1}{2}(3-\sqrt{11})\qquad \text{or} \qquad x_1 \geq \frac{1}{2}(3+\sqrt{11}).
\end{equation}
Additionally, aggregating inequalities \eqref{eq:Texample} with multipliers $(1,0,1)$, we obtain
\[2x_2^2 + x_3^2 + 6x_1 \leq 2\]
from where we conclude $x_1 \leq 1/3$, therefore the leftmost inequality in \eqref{eq:twocases} is valid for $T$, which implies $x_1 \leq 0$. Actually, one can prove that $\frac{1}{2}(3-\sqrt{11}) < 0$ is the optimal value of $x_1$, but a bound suffices for our argument.

On the other hand, the SDP relaxation of the optimization problem reads
\begin{align*}
    \max \quad & x_1 \\
    \text{s.t.} \quad &X_{11} + X_{22} \leq 2 \\
    &-X_{11} - X_{22} \leq -1 \\
    & -X_{11} + X_{22} + X_{33} + 6x_1 \leq 0 \\
    & \begin{bmatrix}
    1 & x^\top \\ 
    x & X 
    \end{bmatrix} \succeq 0.
\end{align*}
And one can easily verify that
\[
\begin{bmatrix}
    1 & x^\top \\ 
    x & X 
    \end{bmatrix} = \begin{bmatrix}
1 & \frac{1}{3} & 0 & 0 \\
\frac{1}{3} & 2& 0 & 0 \\
0 & 0 & 0 & 0 \\
0 & 0 & 0 & 0 
\end{bmatrix} \succeq 0
\]
is a feasible solution with objective value $1/3 > 0$, thus showing that the SDP relaxation is not tight.
%

\subsection{Proof of \Cref{example:4}} \label{sec:infiniteaggproof}
In this subsection, we show that in \Cref{example:4} the convex hull can be obtained via aggregations, but only using infinitely many of them.
\infiniteaggregations*
It is easy to see that $$\textup{conv}(\genericquadopen) = \{ x \in \mathbb{R}^2:  x_1^2 \leq 1,\  x_2^2 \leq 1, \ x_1 + x_2 \leq 1 \}.$$ 
See \Cref{fig:infiniteaggregation} for an illustration of the set $S$ and its convex hull.
Additionally, \PDLC{} condition does not hold. Indeed, \PDLC{} condition holds if and only if 0 is the optimal value of the following SDP:
\begin{eqnarray*}
&\textup{min }& \theta_4 \\
&\textup{s.t.} & \theta_1\left[\begin{array}{ccc} 1& 0 & 0\\ 0 & 0& 0 \\ 0 & 0 & -1\end{array}\right]+ \theta_2\left[\begin{array}{ccc} 0& 0 & 0\\ 0 & 1& 0 \\ 0 & 0 & -1\end{array}\right] + \theta_3\left[\begin{array}{ccc} -1& 0 & 1\\ 0 & -1 & 1 \\ 1 & 1 & -1\end{array}\right]+ \theta_4 \cdot I\succeq I \\
&& \theta_4 \geq 0,
\end{eqnarray*}
where $I$ is the identity matrix. 
The dual of this SDP is:
\begin{eqnarray*}
&\textup{max}\,& \langle I, W\rangle \\
&\textup{s.t.}\, & \left\langle W, \left[\begin{array}{ccc} 1& 0 & 0\\ 0 & 0& 0 \\ 0 & 0 & -1\end{array}\right] \right\rangle  = 0\\ 
&& \left \langle W, \left[\begin{array}{ccc} 0& 0 & 0\\ 0 & 1& 0 \\ 0 & 0 & -1\end{array}\right]\right \rangle = 0 \\
&& \left \langle W, \left[\begin{array}{ccc} -1& 0 & 1\\ 0 & -1 & 1 \\ 1 & 1 & -1\end{array}\right]\right \rangle = 0 \\
&& \left \langle W, I \right \rangle \leq 1\\
&& W \succeq 0. 
\end{eqnarray*}

Note that $$W = \left[\begin{array}{ccc} 1/3 & 1/4 & 1/4 \\ 1/4 & 1/3 & 1/4 \\ 1/4& 1/4 & 1/3 \end{array} \right]$$ is a feasible solution to the dual and its objective function value is $1$. By weak duality we conclude that the \PDLC{} condition does not hold.

Recall that $\Omega^+ := \{ \lambda \in \mathbb{R}^3_{+}: \genericquadopen_{\lambda} \supseteq \textup{conv}(\genericquadopen)\}$. We now show that 
\[\conv(\genericquadopen) = \bigcap_{\lambda \in \Omega^+} \genericquadopen_{\lambda}.\] 
By definition, we have, $\conv(\genericquadopen) \subseteq \bigcap_{\lambda \in \Omega^+} \genericquadopen_{\lambda}$. We need to verify $\conv(\genericquadopen) \supseteq \bigcap_{\lambda \in \Omega^+} \genericquadopen_{\lambda}$. Since the constraints $x_1^2 \leq 1 $ and $x^2_2 \leq 1$ contain $\conv(\genericquadopen)$, it is sufficient to show that for each point in the set $\{ (x_1, x_2): x_1 + x_2 > 1, x_1 <1, x_2 <1\}$ there is a $\lambda \in \Omega^+$ that \emph{does not} contain it. 

Consider $\lambda^a  = \left(a^2,a^2-2 a+1,a^2-a+1\right)$ for $a \in (0,1)$. It is not difficult to verify that $\lambda^a \geq 0$ and
$$\genericquadopen_{\lambda^a}:=\left\{ x: (a-1) x_1^2 - a x_2^2 + 2 \left(a^2-a+1\right) (x_1+x_2)-3 a^2 + 3 a-2 \leq 0 \right\}$$ 

Let us show that $\genericquadopen_{\lambda^a} \supseteq \conv(\genericquadopen)$. Let
\[ g_a(x) \coloneqq (a-1) x_1^2 - a x_2^2 + 2 \left(a^2-a+1\right) (x_1+x_2)-3 a^2 + 3 a-2 \]
Penalizing the constraint $x_1 + x_2 \leq 1$ with the multiplier $2+4(a-1)a \geq 0$ we obtain:
\begin{align*}
    \max_{x\in \conv(S)}\, g_a(x) & \leq \max_{x}\, g_a(x) + (2+4(a-1)a)(1-x_1-x_2) \\
    &= \max_x\, (a-1) \left(-2 a x_1 + a + x_1^2\right)-a x_2^2-2 (a-1) a x_2
\end{align*}
The last function is \emph{strictly concave} for $a\in (0,1)$, and its unique optimal solution is $x=(a,1-a)$, with optimal value $0$. This shows that $\genericquadopen_{\lambda^a} \supseteq \conv(\genericquadopen)$.

In addition, it is not difficult to verify that all points of the form $(a + \epsilon, 1 - a + \epsilon)$ for $\epsilon >0$ are no contained in $\genericquadopen_{\lambda^a}$. Therefore, 
$$\conv(\genericquadopen) = \{x : x_1^2\leq1, x_2^2 \leq 1\}  \cap \bigcap_{a \in (0,1)} S_{\lambda^a}.$$

Finally, we will verify that we need infinitely many aggregations. Consider any $\lambda \in \mathbb{R}^3_{+}\setminus \{0\}$ such that the inequality defining $\genericquadopen_{\lambda}$ is active at $(a, 1 - a)$ and valid for $\conv(\genericquadopen)$. Then $\lambda$ must satisfy:
\begin{eqnarray*}
(a^2 - 1) \lambda_1 + (a^2 - 2a) \lambda_2 + (-2a^2 + 2a) \lambda_3  &=& 0 
\end{eqnarray*}
Since  $a^2 -1 < 0$, $a^2 - 2a <0$, and $-2a^2 + 2a > 0$ it must be that $\lambda_3 >0$. Additionally, since $|a^2 - 1| > 2a - 2a^2$ and $a^2  - 2a < 0$, we must have that $\lambda_1 < \lambda_3$. Similarly, since $|a^2 - 2a| > 2a - 2a^2$, we must have that $\lambda_2 < \lambda_3$. 

Note now that
$$\genericquadopen_{\lambda}:= \{x: (\lambda_1 - \lambda_3) x_1^2 + (\lambda_2 - \lambda_3)x_2^2 + 2\lambda_3(x + y) - \lambda_1 -\lambda_2 -\lambda_3 \leq 0\}.$$
Since we just showed $\lambda_i < \lambda_3$ for $i=1,2$, the last set is the complement of an ellipse. From here we conclude that if $(b,1-b)$ for $b\neq a$ is also in the boundary of $\genericquadopen_{\lambda}$ then $\genericquadopen_{\lambda}$ does not contain $\conv(\genericquadopen)$. 

This shows that any aggregation that has $(a,1-a)$ in its boundary, does not contain any other point $(b,1-b)$.

\bibliographystyle{plain}
\bibliography{references}

\begin{thebibliography}{10}

\bibitem{anstreicher2017kronecker}
Kurt~M. Anstreicher.
\newblock Kronecker product constraints with an application to the
  two-trust-region subproblem.
\newblock {\em SIAM Journal on Optimization}, 27(1):368--378, 2017.

\bibitem{argue2020necessary}
Charles~J. Argue, Fatma K{\i}l{\i}n{\c{c}}-Karzan, and Alex~L. Wang.
\newblock Necessary and sufficient conditions for rank-one generated cones.
\newblock {\em arXiv preprint arXiv:2007.07433}, 2020.

\bibitem{barvinok2001remark}
Alexander Barvinok.
\newblock A remark on the rank of positive semidefinite matrices subject to
  affine constraints.
\newblock {\em Discrete \& Computational Geometry}, 25(1):23--31, 2001.

\bibitem{belotti2013families}
Pietro Belotti, Julio~C. G{\'o}ez, Imre P{\'o}lik, Ted~K. Ralphs, and Tam{\'a}s
  Terlaky.
\newblock On families of quadratic surfaces having fixed intersections with two
  hyperplanes.
\newblock {\em Discrete Applied Mathematics}, 161(16-17):2778--2793, 2013.

\bibitem{NemirovskiNotes}
Aharon Ben-Tal and Arkadi Nemirovski.
\newblock {\em Lectures on Modern Convex Optimization}.
\newblock Society for Industrial and Applied Mathematics, 2001.

\bibitem{bienstock2020outer}
Daniel Bienstock, Chen Chen, and Gonzalo Munoz.
\newblock Outer-product-free sets for polynomial optimization and oracle-based
  cuts.
\newblock {\em Mathematical Programming}, 183(1):105--148, 2020.

\bibitem{bodur2018aggregation}
Merve Bodur, Alberto Del~Pia, Santanu~S Dey, Marco Molinaro, and Sebastian
  Pokutta.
\newblock Aggregation-based cutting-planes for packing and covering integer
  programs.
\newblock {\em Mathematical Programming}, 171(1):331--359, 2018.

\bibitem{burer2015gentle}
Samuel Burer.
\newblock A gentle, geometric introduction to copositive optimization.
\newblock {\em Mathematical Programming}, 151(1):89--116, 2015.

\bibitem{burer2013second}
Samuel Burer and Kurt~M. Anstreicher.
\newblock Second-order-cone constraints for extended trust-region subproblems.
\newblock {\em SIAM Journal on Optimization}, 23(1):432--451, 2013.

\bibitem{Burer2016}
Samuel Burer and Fatma K{\i}l{\i}n{\c{c}}-Karzan.
\newblock How to convexify the intersection of a second order cone and a
  nonconvex quadratic.
\newblock {\em Mathematical Programming}, pages 1--37, 2016.

\bibitem{burer2015trust}
Samuel Burer and Boshi Yang.
\newblock The trust region subproblem with non-intersecting linear constraints.
\newblock {\em Mathematical Programming}, 149(1):253--264, 2015.

\bibitem{burer2019exact}
Samuel Burer and Yinyu Ye.
\newblock Exact semidefinite formulations for a class of (random and
  non-random) nonconvex quadratic programs.
\newblock {\em Mathematical Programming}, pages 1--17, 2019.

\bibitem{dey2019convexifications}
Santanu~S. Dey, Burak Kocuk, and Asteroide Santana.
\newblock Convexifications of rank-one-based substructures in {QCQPs} and
  applications to the pooling problem.
\newblock {\em Journal of Global Optimization}, pages 1--46, 2019.

\bibitem{DeySantana2018}
Santanu~S. Dey, Asteroide Santana, and Yang Wang.
\newblock New {SOCP} relaxation and branching rule for bipartite bilinear
  programs.
\newblock {\em Optimization and Engineering}, Sep 2018.

\bibitem{GDR2021}
Xiaoyi Gu, Santanu~S. Dey, and Jean-Philippe~P. Richard.
\newblock Lifting convex inequalities for bipartite bilinear programs.
\newblock {\em IPCO}, 2021.

\bibitem{Horn1985matrix}
Roger Horn and Charles Johnson.
\newblock {\em Matrix analysis}.
\newblock Cambridge University Press, 1985.

\bibitem{Modaresi2017}
Sina Modaresi and Juan~Pablo Vielma.
\newblock Convex hull of two quadratic or a conic quadratic and a quadratic
  inequality.
\newblock {\em Mathematical Programming}, 164(1):383--409, Jul 2017.

\bibitem{muller2020generalized}
Benjamin M{\"u}ller, Gonzalo Mu{\~n}oz, Maxime Gasse, Ambros Gleixner, Andrea
  Lodi, and Felipe Serrano.
\newblock On generalized surrogate duality in mixed-integer nonlinear
  programming.
\newblock In {\em International Conference on Integer Programming and
  Combinatorial Optimization}, pages 322--337. Springer, Cham, 2020.

\bibitem{munoz2020maximal}
Gonzalo Mu{\~n}oz and Felipe Serrano.
\newblock Maximal quadratic-free sets.
\newblock In {\em International Conference on Integer Programming and
  Combinatorial Optimization}, pages 307--321. Springer, 2020.

\bibitem{pataki1998rank}
G{\'a}bor Pataki.
\newblock On the rank of extreme matrices in semidefinite programs and the
  multiplicity of optimal eigenvalues.
\newblock {\em Mathematics of operations research}, 23(2):339--358, 1998.

\bibitem{polik2007survey}
Imre P{\'o}lik and Tam{\'a}s Terlaky.
\newblock A survey of the s-lemma.
\newblock {\em SIAM review}, 49(3):371--418, 2007.

\bibitem{asterquad}
Asteroide Santana and Santanu~S. Dey.
\newblock The convex hull of a quadratic constraint over a polytope.
\newblock {\em SIAM Journal on Optimization}, 30(4):2983--2997, 2020.

\bibitem{tawarmalani2010strong}
Mohit Tawarmalani, Jean-Philippe~P Richard, and Kwanghun Chung.
\newblock Strong valid inequalities for orthogonal disjunctions and bilinear
  covering sets.
\newblock {\em Mathematical Programming}, 124(1):481--512, 2010.

\bibitem{wang2020convex}
Alex~L. Wang and Fatma K{\i}l{\i}n{\c{c}}-Karzan.
\newblock On convex hulls of epigraphs of {QCQP}s.
\newblock In {\em International Conference on Integer Programming and
  Combinatorial Optimization}, pages 419--432. Springer, 2020.

\bibitem{wang2021tightness}
Alex~L. Wang and Fatma K{\i}l{\i}n{\c{c}}-Karzan.
\newblock On the tightness of {SDP} relaxations of {QCQP}s.
\newblock {\em Mathematical Programming}, pages 1--41, 2021.

\bibitem{yakubovich1977}
Vladimir~Andreevich Yakubovich.
\newblock S-procedure in nonlinear control theory.
\newblock {\em Vestnick Leningrad Univ. Math.}, 4:73--93, 1977.

\bibitem{ye2003new}
Yinyu Ye and Shuzhong Zhang.
\newblock New results on quadratic minimization.
\newblock {\em SIAM Journal on Optimization}, 14(1):245--267, 2003.

\bibitem{yildiran2009convex}
U{\u{g}}ur Yildiran.
\newblock Convex hull of two quadratic constraints is an {LMI} set.
\newblock {\em IMA Journal of Mathematical Control and Information},
  26(4):417--450, 2009.

\end{thebibliography}
\end{document}